\documentclass{tac}  

\usepackage{mathrsfs,amssymb,stmaryrd,bbm,mathtools, upgreek}
\usepackage{chngcntr}
\usepackage[shortlabels]{enumitem}
\makeatletter
\@addtoreset{equation}{section}
\makeatother
\numberwithin{equation}{subsection}

\newtheorem{theo}{Theorem}[section]
\newtheorem{lem}[theo]{Lemma}
\newtheorem{prop}[theo]{Proposition}
\newtheorem{coro}[theo]{Corollary}
\newtheorem{defi}[theo]{Definition}
\newtheorem{rem}[theo]{Remark}

\let\olddefi\defi
\renewcommand{\defi}{\olddefi\normalfont}
\let\oldrem\rem
\renewcommand{\rem}{\oldrem\upshape}

\def\mathrmdef#1{\expandafter\def\csname#1\endcsname{{\rm#1}}}
\def\mathsfdef#1{\expandafter\def\csname#1\endcsname{{\rm\sf#1}}}
\def\mathcaldef#1{\expandafter\def\csname#1\endcsname{{\mathcal#1}}}

\mathrmdef{id}
\mathrmdef{op}
\mathrmdef{co}
 \mathsfdef{Top} 
 \mathsfdef{Set} 
 \mathsfdef{Cat}
  \mathsfdef{cat}
 \def\Catt{ \underline{\mathsf{Cat}} }

\mathsfdef{colim}
\mathrmdef{obj}
\mathsfdef{lim}

\def\ran{\mathrm{Ran} }
\def\lan{\mathrm{Lan} }

\def\dd{\mathbbmss{d}}
\def\A{\mathfrak{A}}

\def\b{\mathtt{B}}
\def\cC{\mathtt{C}}
\def\d{\mathtt{D}}
\def\s{\mathtt{S}}

\def\j{\mathrm{j}}

\def\laxdea{\mathrm{lax}\textrm{-}\mathcal{D}\mathrm{esc}
\left(\AAAA \right) }

\def\laxde{\mathrm{lax}\textrm{-}\mathcal{D}\mathrm{esc}}

\def\aaaa{\mathfrak{a}}
\def\bbbb{\mathfrak{b}}
\def\ffff{\mathfrak{f}}

\def\AAAA{\mathcal{A}}
\def\BBBB{\mathcal{B}}

\def\FFFF{\mathcal{F}}
\def\GGGG{\mathcal{G}}

\def\SSSSS{\mathbb{S}}

\def\BBBBB{\mathbb{B}}
\def\CCCCC{\mathbb{C}}

\usepackage{ifluatex}
\input diagxy        
\xyoption{curve}     
\xyoption{all}
\xyoption{2cell}

\begin{document}  
	
\ifluatex
\catcode`\^^J=10
\directlua{dofile "dednat6load.lua"}
\else
%
\def\diagxyto{\ifnextchar/{\toop}{\toop/>/}}
\def\to     {\rightarrow}
\def\defded#1#2{\expandafter\def\csname ded-#1\endcsname{#2}}
\def\ifdedundefined#1{\expandafter\ifx\csname ded-#1\endcsname\relax}
\def\ded#1{\ifdedundefined{#1}
    \errmessage{UNDEFINED DEDUCTION: #1}
  \else
    \csname ded-#1\endcsname
  \fi
}
\def\defdiag#1#2{\expandafter\def\csname diag-#1\endcsname{\bfig#2\efig}}
\def\defdiagprep#1#2#3{\expandafter\def\csname diag-#1\endcsname{{#2\bfig#3\efig}}}
\def\ifdiagundefined#1{\expandafter\ifx\csname diag-#1\endcsname\relax}
\def\diag#1{\ifdiagundefined{#1}
    \errmessage{UNDEFINED DIAGRAM: #1}
  \else
    \csname diag-#1\endcsname
  \fi
}
\newlength{\celllower}
\newlength{\lcelllower}
\def\cellfont{}
\def\lcellfont{}
\def\cell #1{\lower\celllower\hbox to 0pt{\hss\cellfont${#1}$\hss}}
\def\lcell#1{\lower\celllower\hbox to 0pt   {\lcellfont${#1}$\hss}}
\def\expr#1{\directlua{output(tostring(#1))}}
\def\eval#1{\directlua{#1}}
\def\pu{\directlua{pu()}}
%

\defdiag{underlyingprecategory}{   
  \morphism(0,0)|m|/->/<1200,0>[{\Catt\left(\mathsf{1},a\right)}`{\Catt\left(\mathsf{2},a\right)};{\Catt\left(s^0,a\right)}]
  \morphism(1200,0)|a|/{@{->}@/_30pt/}/<-1200,0>[{\Catt\left(\mathsf{2},a\right)}`{\Catt\left(\mathsf{1},a\right)};{\Catt\left(d^0,a\right)}]
  \morphism(1200,0)|b|/{@{->}@/^30pt/}/<-1200,0>[{\Catt\left(\mathsf{2},a\right)}`{\Catt\left(\mathsf{1},a\right)};{\Catt\left(d^1,a\right)}]
  \morphism(2400,0)|m|/->/<-1200,0>[{\Catt\left(\mathsf{3},a\right)}`{\Catt\left(\mathsf{2},a\right)};{\Catt\left(D^1,a\right)}]
  \morphism(2400,0)|a|/{@{->}@/_30pt/}/<-1200,0>[{\Catt\left(\mathsf{3},a\right)}`{\Catt\left(\mathsf{2},a\right)};{\Catt\left(D^0,a\right)}]
  \morphism(2400,0)|b|/{@{->}@/^30pt/}/<-1200,0>[{\Catt\left(\mathsf{3},a\right)}`{\Catt\left(\mathsf{2},a\right)};{\Catt\left(D^2,a\right)}]
}
\defdiag{underlyingprecategoryofamonoid}{   
  \morphism(0,0)|m|/->/<1200,0>[{\left\{m\right\}}`{m};{\Sigma{m(s_0)}}]
  \morphism(1200,0)/{@{->}@/_20pt/}/<-1200,0>[{m}`{\left\{m\right\}};]
  \morphism(1200,0)/{@{->}@/^20pt/}/<-1200,0>[{m}`{\left\{m\right\}};]
  \morphism(2400,0)|m|/->/<-1200,0>[{m\times{m}}`{m};{\Sigma{m}(D_1)}]
  \morphism(2400,0)|a|/{@{->}@/_20pt/}/<-1200,0>[{m\times{m}}`{m};{\Sigma{m}(D_0)}]
  \morphism(2400,0)|b|/{@{->}@/^20pt/}/<-1200,0>[{m\times{m}}`{m};{\Sigma{m}(D_2)}]
}
\defdiag{descentunderlyingprecategory}{   
  \morphism(1200,0)|m|/->/<-1200,0>[{\Catt\left(\mathsf{2},a\right)}`{\Catt\left(\mathsf{1},a\right)};{\Catt\left(s^0,a\right)^\ast}]
  \morphism(0,0)|b|/{@{->}@/_30pt/}/<1200,0>[{\Catt\left(\mathsf{1},a\right)}`{\Catt\left(\mathsf{2},a\right)};{\Catt\left(d^1,a\right)^\ast}]
  \morphism(0,0)|a|/{@{->}@/^30pt/}/<1200,0>[{\Catt\left(\mathsf{1},a\right)}`{\Catt\left(\mathsf{2},a\right)};{\Catt\left(d^0,a\right)^\ast}]
  \morphism(1200,0)|m|/->/<1200,0>[{\Catt\left(\mathsf{2},a\right)}`{\Catt\left(\mathsf{3},a\right)};{\Catt\left(D^1,a\right)^\ast}]
  \morphism(1200,0)|b|/{@{->}@/_30pt/}/<1200,0>[{\Catt\left(\mathsf{2},a\right)}`{\Catt\left(\mathsf{3},a\right)};{\Catt\left(D^2,a\right)^\ast}]
  \morphism(1200,0)|a|/{@{->}@/^30pt/}/<1200,0>[{\Catt\left(\mathsf{2},a\right)}`{\Catt\left(\mathsf{3},a\right)};{\Catt\left(D^0,a\right)^\ast}]
}
\defdiag{internalactionsdescentunderlyingprecategory}{   
  \morphism(1200,0)|m|/->/<-1200,0>[{\mathcal{F}\,{a(\mathsf{2})}}`{\mathcal{F}\,{a(\mathsf{1})}};{\mathcal{F}\,{a(s_0)}}]
  \morphism(0,0)|b|/{@{->}@/_30pt/}/<1200,0>[{\mathcal{F}\,{a(\mathsf{1})}}`{\mathcal{F}\,{a(\mathsf{2})}};{\mathcal{F}\,{a(d_1)}}]
  \morphism(0,0)|a|/{@{->}@/^30pt/}/<1200,0>[{\mathcal{F}\,{a(\mathsf{1})}}`{\mathcal{F}\,{a(\mathsf{2})}};{\mathcal{F}\,{a(d_0)}}]
  \morphism(1200,0)|m|/->/<1200,0>[{\mathcal{F}\,{a(\mathsf{2})}}`{\mathcal{F}\,{a(\mathsf{3})}};{\mathcal{F}\,{a(D_1)}}]
  \morphism(1200,0)|b|/{@{->}@/_30pt/}/<1200,0>[{\mathcal{F}\,{a(\mathsf{2})}}`{\mathcal{F}\,{a(\mathsf{3})}};{\mathcal{F}\,{a(D_2)}}]
  \morphism(1200,0)|a|/{@{->}@/^30pt/}/<1200,0>[{\mathcal{F}\,{a(\mathsf{2})}}`{\mathcal{F}\,{a(\mathsf{3})}};{\mathcal{F}\,{a(D_0)}}]
}
\defdiag{pseudofunctorfromdeltatres}{   
  \morphism(1200,0)|m|/->/<-1200,0>[{\AAAA(\mathsf{2})}`{\AAAA(\mathsf{1})};{\AAAA(s^0)}]
  \morphism(0,0)|b|/{@{->}@/_30pt/}/<1200,0>[{\AAAA(\mathsf{1})}`{\AAAA(\mathsf{2})};{\AAAA(d^1)}]
  \morphism(0,0)|a|/{@{->}@/^30pt/}/<1200,0>[{\AAAA(\mathsf{1})}`{\AAAA(\mathsf{2})};{\AAAA(d^0)}]
  \morphism(1200,0)|m|/->/<1200,0>[{\AAAA(\mathsf{2})}`{\AAAA(\mathsf{3})};{\AAAA(D^1)}]
  \morphism(1200,0)|b|/{@{->}@/_30pt/}/<1200,0>[{\AAAA(\mathsf{2})}`{\AAAA(\mathsf{3})};{\AAAA(D^2)}]
  \morphism(1200,0)|a|/{@{->}@/^30pt/}/<1200,0>[{\AAAA(\mathsf{2})}`{\AAAA(\mathsf{3})};{\AAAA(D^0)}]
}
\defdiag{cosimplicial_onetwothree}{   
  \morphism(600,0)|m|/->/<-600,0>[{\mathsf{2}}`{\mathsf{1}};{s^0}]
  \morphism(0,0)|b|/{@{->}@/_20pt/}/<600,0>[{\mathsf{1}}`{\mathsf{2}};{d^1}]
  \morphism(0,0)|a|/{@{->}@/^20pt/}/<600,0>[{\mathsf{1}}`{\mathsf{2}};{d^0}]
  \morphism(600,0)|m|/->/<600,0>[{\mathsf{2}}`{\mathsf{3}};{D^1}]
  \morphism(600,0)|b|/{@{->}@/_20pt/}/<600,0>[{\mathsf{2}}`{\mathsf{3}};{D^2}]
  \morphism(600,0)|a|/{@{->}@/^20pt/}/<600,0>[{\mathsf{2}}`{\mathsf{3}};{D^0}]
}
\defdiag{universaltwocelllaxdescentpsi}{   
  \morphism(450,0)|a|/->/<-450,-375>[{\laxdea}`{\AAAA{(\mathsf{1})}};{\dd^\AAAA}]
  \morphism(450,0)|a|/->/<450,-375>[{\laxdea}`{\AAAA{(\mathsf{1})}};{\dd^\AAAA}]
  \morphism(0,-375)|b|/->/<450,-375>[{\AAAA{(\mathsf{1})}}`{\AAAA{(\mathsf{2})}};{\AAAA(d^1)}]
  \morphism(900,-375)|b|/->/<-450,-375>[{\AAAA{(\mathsf{1})}}`{\AAAA{(\mathsf{2})}};{\AAAA(d^0)}]
  \morphism(225,-375)|a|/=>/<450,0>[{\phantom{O}}`{\phantom{O}};{\uppsi}]
}
\defdiag{universalonecelllaxdescentda}{   
  \morphism(0,0)|a|/->/<675,0>[{\laxdea}`{\AAAA{(\mathsf{1})}};{\dd^\AAAA}]
}
\defdiag{laxdescentassociativityleftside}{   
  \morphism(300,0)|a|/->/<-300,-450>[{\AAAA{(\mathsf{1})}}`{\AAAA{(\mathsf{2})}};{\AAAA(d^1)}]
  \morphism(300,0)|m|/->/<300,-450>[{\AAAA{(\mathsf{1})}}`{\AAAA{(\mathsf{2})}};{\AAAA(d^1)}]
  \morphism(900,0)|a|/->/<-600,0>[{\s}`{\AAAA{(\mathsf{1})}};{F}]
  \morphism(900,0)|r|/->/<300,-450>[{\s}`{\AAAA{(\mathsf{1})}};{F}]
  \morphism(0,-450)|b|/->/<300,-450>[{\AAAA{(\mathsf{2})}}`{\AAAA{(\mathsf{3})}};{\AAAA(D^2)}]
  \morphism(600,-450)|m|/->/<-300,-450>[{\AAAA{(\mathsf{2})}}`{\AAAA{(\mathsf{3})}};{\AAAA(D^1)}]
  \morphism(1200,-450)|m|/->/<-600,0>[{\AAAA{(\mathsf{1})}}`{\AAAA{(\mathsf{2})}};{\AAAA(d^0)}]
  \morphism(1200,-450)|r|/->/<-300,-450>[{\AAAA{(\mathsf{1})}}`{\AAAA{(\mathsf{2})}};{\AAAA(d^0)}]
  \morphism(900,-900)|b|/->/<-600,0>[{\AAAA{(\mathsf{2})}}`{\AAAA{(\mathsf{3})}};{\AAAA(D^0)}]
  \morphism(112,-450)|a|/=>/<375,0>[{\phantom{O}}`{\phantom{O}};{\AAAA(\sigma_{12})}]
  \morphism(548,-225)|a|/=>/<450,0>[{\phantom{O}}`{\phantom{O}};{\beta}]
  \morphism(562,-675)|a|/=>/<450,0>[{\phantom{O}}`{\phantom{O}};{\AAAA(\sigma_{01})}]
}
\defdiag{laxdescentassociativityrightside}{   
  \morphism(900,0)|a|/->/<300,-450>[{\s}`{\AAAA{(\mathsf{1})}};{F}]
  \morphism(900,0)|m|/->/<-300,-450>[{\s}`{\AAAA{(\mathsf{1})}};{F}]
  \morphism(300,0)|a|/<-/<600,0>[{\AAAA{(\mathsf{1})}}`{\s};{F}]
  \morphism(300,0)|l|/->/<-300,-450>[{\AAAA{(\mathsf{1})}}`{\AAAA{(\mathsf{2})}};{\AAAA(d^1)}]
  \morphism(1200,-450)|b|/->/<-300,-450>[{\AAAA{(\mathsf{1})}}`{\AAAA{(\mathsf{2})}};{\AAAA(d^0)}]
  \morphism(600,-450)|m|/->/<300,-450>[{\AAAA{(\mathsf{1})}}`{\AAAA{(\mathsf{2})}};{\AAAA(d^1)}]
  \morphism(0,-450)|m|/<-/<600,0>[{\AAAA{(\mathsf{2})}}`{\AAAA{(\mathsf{1})}};{\AAAA(d^0)}]
  \morphism(0,-450)|l|/->/<300,-450>[{\AAAA{(\mathsf{2})}}`{\AAAA{(\mathsf{3})}};{\AAAA(D^2)}]
  \morphism(300,-900)|b|/<-/<600,0>[{\AAAA{(\mathsf{3})}}`{\AAAA{(\mathsf{2})}};{\AAAA(D^0)}]
  \morphism(712,-450)|a|/=>/<375,0>[{\phantom{O}}`{\phantom{O}};{\beta}]
  \morphism(202,-225)|a|/=>/<450,0>[{\phantom{O}}`{\phantom{O}};{\beta}]
  \morphism(188,-675)|a|/=>/<450,0>[{\phantom{O}}`{\phantom{O}};{\AAAA(\sigma_{02})}]
}
\defdiag{laxdescentidentityleftside}{   
  \morphism(0,0)|a|/->/<525,0>[{\s}`{\AAAA{(\mathsf{1})}};{F}]
  \morphism(0,0)|l|/->/<0,-525>[{\s}`{\AAAA{(\mathsf{1})}};{F}]
  \morphism(525,0)|m|/->/<0,-525>[{\AAAA{(\mathsf{1})}}`{\AAAA{(\mathsf{2})}};{\AAAA(d^0)}]
  \morphism(525,0)/{@{=}@/^22pt/}/<300,-900>[{\AAAA{(\mathsf{1})}}`{\AAAA{(\mathsf{1})}};]
  \morphism(0,-525)|m|/->/<525,0>[{\AAAA{(\mathsf{1})}}`{\AAAA{(\mathsf{2})}};{\AAAA(d^1)}]
  \morphism(0,-525)/{@{=}@/_20pt/}/<825,-375>[{\AAAA{(\mathsf{1})}}`{\AAAA{(\mathsf{1})}};]
  \morphism(525,-525)|m|/->/<300,-375>[{\AAAA{(\mathsf{2})}}`{\AAAA{(\mathsf{1})}};{\AAAA(s^0)}]
  \morphism(150,-712)|a|/{@{=>}@<-3pt>}/<450,0>[{\phantom{O}}`{\phantom{O}};{\AAAA(\mathfrak{n}_1)^{-1}}]
  \morphism(540,-450)|a|/=>/<345,0>[{\phantom{O}}`{\phantom{O}};{\AAAA(\mathfrak{n}_0)}]
  \morphism(90,-262)|a|/=>/<345,0>[{\phantom{O}}`{\phantom{O}};{\beta}]
}
\defdiag{laxdescentidentityrightside}{   
  \morphism(0,0)|r|/{@{->}@/^13pt/}/<0,-900>[{\s}`{\AAAA{(\mathsf{1})}};{F}]
  \morphism(0,0)|l|/{@{->}@/_13pt/}/<0,-900>[{\s}`{\AAAA{(\mathsf{1})}};{F}]
  \morphism(-98,-450)/=/<195,0>[{\phantom{O}}`{\phantom{O}};]
}
\defdiag{equation_on_the_descent_datum_leftside}{   
  \morphism(450,0)|r|/->/<0,-300>[{\s}`{\laxdea};{\check{F}}]
  \morphism(450,-300)|l|/->/<-450,-300>[{\laxdea}`{\AAAA{(\mathsf{1})}};{\dd^{\AAAA}}]
  \morphism(450,-300)|r|/->/<450,-300>[{\laxdea}`{\AAAA{(\mathsf{1})}};{\dd^{\AAAA}}]
  \morphism(0,-600)|l|/->/<450,-300>[{\AAAA{(\mathsf{1})}}`{\AAAA{(\mathsf{2})}};{\AAAA(d^1)}]
  \morphism(900,-600)|r|/->/<-450,-300>[{\AAAA{(\mathsf{1})}}`{\AAAA{(\mathsf{2})}};{\AAAA(d^0)}]
  \morphism(188,-600)|a|/=>/<525,0>[{\phantom{O}}`{\phantom{O}};{\uppsi}]
}
\defdiag{equation_on_the_descent_datum_rightside}{   
  \morphism(450,0)|l|/->/<-450,-600>[{\s}`{\AAAA{(\mathsf{1})}};{F}]
  \morphism(450,0)|r|/->/<450,-600>[{\s}`{\AAAA{(\mathsf{1})}};{F}]
  \morphism(0,-600)|l|/->/<450,-300>[{\AAAA{(\mathsf{1})}}`{\AAAA{(\mathsf{2})}};{\AAAA(d^1)}]
  \morphism(900,-600)|r|/->/<-450,-300>[{\AAAA{(\mathsf{1})}}`{\AAAA{(\mathsf{2})}};{\AAAA(d^0)}]
  \morphism(188,-450)|a|/=>/<525,0>[{\phantom{O}}`{\phantom{O}};{\beta}]
}
\defdiag{equation_two_cell_for_descent_right_side_2}{   
  \morphism(450,-1350)|l|/<-/<-450,450>[{\AAAA{(\mathsf{2})}}`{\AAAA{(\mathsf{1})}};{\AAAA(d^1)}]
  \morphism(450,-1350)|r|/<-/<450,450>[{\AAAA{(\mathsf{2})}}`{\AAAA{(\mathsf{1})}};{\AAAA(d^0)}]
  \morphism(0,-900)|m|/<-/<450,450>[{\AAAA{(\mathsf{1})}}`{\laxdea};{\dd^{\AAAA}}]
  \morphism(900,-900)|m|/<-/<-450,450>[{\AAAA{(\mathsf{1})}}`{\laxdea};{\dd^{\AAAA}}]
  \morphism(450,-450)|m|/<-/<0,450>[{\laxdea}`{\s};{F_1'}]
  \morphism(900,-900)|r|/{@{<-}@/_30pt/}/<-450,900>[{\AAAA{(\mathsf{1})}}`{\s};{\dd^{\AAAA}\,\circ\,F_0'}]
  \morphism(225,-900)|a|/=>/<450,0>[{\phantom{O}}`{\phantom{O}};{\uppsi}]
  \morphism(675,-450)|a|/=>/<300,0>[{\phantom{O}}`{\phantom{O}};{\xi}]
}
\defdiag{equation_two_cell_for_descent_left_side_2}{   
  \morphism(450,-1350)|l|/<-/<-450,450>[{\AAAA{(\mathsf{2})}}`{\AAAA{(\mathsf{1})}};{\AAAA(d^1)}]
  \morphism(450,-1350)|r|/<-/<450,450>[{\AAAA{(\mathsf{2})}}`{\AAAA{(\mathsf{1})}};{\AAAA(d^0)}]
  \morphism(0,-900)|m|/<-/<450,450>[{\AAAA{(\mathsf{1})}}`{\laxdea};{\dd^{\AAAA}}]
  \morphism(900,-900)|m|/<-/<-450,450>[{\AAAA{(\mathsf{1})}}`{\laxdea};{\dd^{\AAAA}}]
  \morphism(450,-450)|m|/<-/<0,450>[{\laxdea}`{\s};{F_0'}]
  \morphism(0,-900)|l|/{@{<-}@/^30pt/}/<450,900>[{\AAAA{(\mathsf{1})}}`{\s};{\dd^{\AAAA}\,\circ\,F_1'}]
  \morphism(225,-900)|a|/=>/<450,0>[{\phantom{O}}`{\phantom{O}};{\uppsi}]
  \morphism(-75,-450)|a|/=>/<300,0>[{\phantom{O}}`{\phantom{O}};{\xi}]
}
\defdiag{xilaxdescenttwocellproperty}{   
  \morphism(0,-675)|l|/{@{<-}@/^30pt/}/<0,675>[{\laxdea}`{\s};{F_1'}]
  \morphism(0,-675)|r|/{@{<-}@/_30pt/}/<0,675>[{\laxdea}`{\s};{F_0'}]
  \morphism(0,-1050)|l|/<-/<0,375>[{\AAAA{(\mathsf{1})}}`{\laxdea};{\dd^{\AAAA}}]
  \morphism(-225,-338)|a|/=>/<450,0>[{\phantom{O}}`{\phantom{O}};{\xi{'}}]
}
\defdiag{definition_of_kan_extension_rightside}{   
  \morphism(0,0)|m|/->/<600,-600>[{\b}`{\cC};{\ran_H\,{J}}]
  \morphism(0,0)|l|/{@{->}@/_30pt/}/<600,-600>[{\b}`{\cC};{R}]
  \morphism(600,0)|a|/->/<-600,0>[{\s}`{\b};{H}]
  \morphism(600,0)|r|/->/<0,-600>[{\s}`{\cC};{J}]
  \morphism(210,-300)|a|/{@{=>}@<18pt>}/<375,0>[{\phantom{O}}`{\phantom{O}};{\gamma}]
  \morphism(75,-300)|a|/{@{=>}@<-18pt>}/<375,0>[{\phantom{O}}`{\phantom{O}};{\beta}]
}
\defdiag{definition_of_kan_extension_leftside}{   
  \morphism(0,0)|m|/->/<600,-600>[{\b}`{\cC};{\ran_H\,{J}}]
  \morphism(0,0)|l|/{@{->}@/_30pt/}/<600,-600>[{\b}`{\cC};{R}]
  \morphism(75,-300)|a|/{@{=>}@<-18pt>}/<375,0>[{\phantom{O}}`{\phantom{O}};{\beta}]
}
\defdiag{preservationofkanextension}{   
  \morphism(0,0)|b|/->/<600,-450>[{\b}`{\cC};{\hat{J}}]
  \morphism(600,0)|a|/->/<-600,0>[{\s}`{\b};{H}]
  \morphism(600,0)|r|/->/<0,-450>[{\s}`{\cC};{J}]
  \morphism(600,-450)|r|/->/<0,-300>[{\cC}`{\d};{G}]
  \morphism(210,-225)|a|/{@{=>}@<13pt>}/<375,0>[{\phantom{O}}`{\phantom{O}};{\gamma}]
}
\defdiag{preservedkanextension}{   
  \morphism(0,0)|b|/->/<600,-600>[{\b}`{\AAAA(\mathsf{1})};{\ran_H\left(\dd{^\AAAA}{J}\right)}]
  \morphism(600,0)|a|/->/<-600,0>[{\s}`{\b};{H}]
  \morphism(600,0)|r|/->/<0,-300>[{\s}`{{\qquad\qquad\laxdea}};{J}]
  \morphism(600,-300)|r|/->/<0,-300>[{{\qquad\qquad\laxdea}}`{\AAAA(\mathsf{1})};{\dd{^\AAAA}}]
  \morphism(600,-600)|r|/->/<0,-300>[{\AAAA(\mathsf{1})}`{\AAAA{(\mathsf{2})}};{\AAAA(d^0)}]
  \morphism(210,-300)|a|/{@{=>}@<13pt>}/<375,0>[{\phantom{O}}`{\phantom{O}};{\nu}]
}
\defdiag{definitionofphileftside}{   
  \morphism(0,0)|l|/->/<0,-900>[{\b}`{\AAAA{(\mathsf{1})}};{\ran_H\left(\dd^\AAAA{J}\right)}]
  \morphism(450,0)|a|/->/<-450,0>[{\s}`{\b};{H}]
  \morphism(450,0)|r|/->/<0,-450>[{\s}`{\laxdea};{J}]
  \morphism(450,-450)|m|/->/<-450,-450>[{\laxdea}`{\AAAA{(\mathsf{1})}};{\dd{^\AAAA}}]
  \morphism(0,-900)|l|/->/<450,-450>[{\AAAA{(\mathsf{1})}}`{\AAAA{(\mathsf{2})}};{\AAAA(d^1)}]
  \morphism(450,-450)|m|/->/<0,-450>[{\laxdea}`{\AAAA{(\mathsf{1})}};{\dd{^\AAAA}}]
  \morphism(450,-900)|r|/->/<0,-450>[{\AAAA{(\mathsf{1})}}`{\AAAA{(\mathsf{2})}};{\AAAA(d^0)}]
  \morphism(38,-225)|a|/=>/<375,0>[{\phantom{O}}`{\phantom{O}};{\nu}]
  \morphism(75,-900)|a|/=>/<300,0>[{\phantom{O}}`{\phantom{O}};{\uppsi}]
}
\defdiag{definitionofphirightside}{   
  \morphism(0,0)|l|/->/<0,-900>[{\b}`{\AAAA{(\mathsf{1})}};{\ran_H\left(\dd^\AAAA{J}\right)}]
  \morphism(1050,0)|a|/->/<-1050,0>[{\s}`{\b};{H}]
  \morphism(1050,0)|r|/->/<0,-450>[{\s}`{\qquad\qquad\laxdea};{J}]
  \morphism(0,0)|m|/->/<1050,-900>[{\b}`{\AAAA{(\mathsf{1})}};{\ran_H\left(\dd^\AAAA{J}\right)}]
  \morphism(0,-900)|m|/->/<1050,-450>[{\AAAA{(\mathsf{1})}}`{\AAAA{(\mathsf{2})}};{\AAAA(d^1)}]
  \morphism(1050,-450)|m|/->/<0,-450>[{\qquad\qquad\laxdea}`{\AAAA{(\mathsf{1})}};{\dd{^\AAAA}}]
  \morphism(1050,-900)|r|/->/<0,-450>[{\AAAA{(\mathsf{1})}}`{\AAAA{(\mathsf{2})}};{\AAAA(d^0)}]
  \morphism(450,-225)|a|/{@{=>}@<-5pt>}/<450,0>[{\phantom{O}}`{\phantom{O}};{\nu}]
  \morphism(112,-675)|a|/=>/<525,0>[{\phantom{O}}`{\phantom{O}};{\varphi}]
}
\defdiag{definitionofvarphilinha}{   
  \morphism(300,0)|a|/->/<-300,-450>[{\AAAA{(\mathsf{1})}}`{\AAAA{(\mathsf{2})}};{\AAAA(d^1)}]
  \morphism(300,0)|m|/->/<300,-450>[{\AAAA{(\mathsf{1})}}`{\AAAA{(\mathsf{2})}};{\AAAA(d^1)}]
  \morphism(1500,0)|r|/->/<300,-450>[{\s}`{\laxdea};{J}]
  \morphism(1500,0)|a|/->/<-600,0>[{\s}`{\b};{H}]
  \morphism(1800,-450)|b|/->/<-600,0>[{\laxdea}`{\AAAA{(\mathsf{1})}};{\dd^\AAAA}]
  \morphism(900,0)|a|/->/<-600,0>[{\b}`{\AAAA{(\mathsf{1})}};{\ran_H\left(\dd^\AAAA{J}\right)}]
  \morphism(900,0)|m|/->/<300,-450>[{\b}`{\AAAA{(\mathsf{1})}};{\ran_H\left(\dd^\AAAA{J}\right)}]
  \morphism(0,-450)|b|/->/<300,-450>[{\AAAA{(\mathsf{2})}}`{\AAAA{(\mathsf{3})}};{\AAAA(D^2)}]
  \morphism(600,-450)|m|/->/<-300,-450>[{\AAAA{(\mathsf{2})}}`{\AAAA{(\mathsf{3})}};{\AAAA(D^1)}]
  \morphism(1200,-450)|m|/->/<-600,0>[{\AAAA{(\mathsf{1})}}`{\AAAA{(\mathsf{2})}};{\AAAA(d^0)}]
  \morphism(1200,-450)|r|/->/<-300,-450>[{\AAAA{(\mathsf{1})}}`{\AAAA{(\mathsf{2})}};{\AAAA(d^0)}]
  \morphism(900,-900)|b|/->/<-600,0>[{\AAAA{(\mathsf{2})}}`{\AAAA{(\mathsf{3})}};{\AAAA(D^0)}]
  \morphism(112,-450)|a|/=>/<375,0>[{\phantom{O}}`{\phantom{O}};{\AAAA(\sigma_{12})}]
  \morphism(502,-225)|a|/=>/<375,0>[{\phantom{O}}`{\phantom{O}};{\varphi}]
  \morphism(562,-675)|a|/=>/<450,0>[{\phantom{O}}`{\phantom{O}};{\AAAA(\sigma_{01})}]
  \morphism(1208,-225)|a|/=>/<435,0>[{\phantom{O}}`{\phantom{O}};{\nu}]
}
\defdiag{definitionofvarphilinhalogoigual}{   
  \morphism(375,-375)|a|/->/<-375,-375>[{\AAAA{(\mathsf{1})}}`{\AAAA{(\mathsf{2})}};{\AAAA(d^1)}]
  \morphism(375,-375)|m|/->/<375,-375>[{\AAAA{(\mathsf{1})}}`{\AAAA{(\mathsf{2})}};{\AAAA(d^1)}]
  \morphism(1125,-375)|a|/->/<-750,0>[{\laxdea}`{\AAAA{(\mathsf{1})}};{\dd^\AAAA}]
  \morphism(1125,-375)|r|/->/<375,-375>[{\laxdea}`{\AAAA{(\mathsf{1})}};{\dd^\AAAA}]
  \morphism(0,-750)|b|/->/<375,-375>[{\AAAA{(\mathsf{2})}}`{\AAAA{(\mathsf{3})}};{\AAAA(D^2)}]
  \morphism(750,-750)|m|/->/<-375,-375>[{\AAAA{(\mathsf{2})}}`{\AAAA{(\mathsf{3})}};{\AAAA(D^1)}]
  \morphism(1500,-750)|m|/->/<-750,0>[{\AAAA{(\mathsf{1})}}`{\AAAA{(\mathsf{2})}};{\AAAA(d^0)}]
  \morphism(1500,-750)|r|/->/<-375,-375>[{\AAAA{(\mathsf{1})}}`{\AAAA{(\mathsf{2})}};{\AAAA(d^0)}]
  \morphism(1125,-1125)|b|/->/<-750,0>[{\AAAA{(\mathsf{2})}}`{\AAAA{(\mathsf{3})}};{\AAAA(D^0)}]
  \morphism(750,0)|a|/->/<-750,0>[{\s}`{\b};{H}]
  \morphism(750,0)|r|/->/<375,-375>[{\s}`{\laxdea};{J}]
  \morphism(0,0)|m|/->/<375,-375>[{\b}`{\AAAA{(\mathsf{1})}};{\ran_H\left(\dd^\AAAA{J}\right)}]
  \morphism(128,-750)|a|/=>/<495,0>[{\phantom{O}}`{\phantom{O}};{\AAAA(\sigma_{12})}]
  \morphism(690,-562)|a|/=>/<495,0>[{\phantom{O}}`{\phantom{O}};{\uppsi}]
  \morphism(690,-938)|a|/=>/<495,0>[{\phantom{O}}`{\phantom{O}};{\AAAA(\sigma_{01})}]
  \morphism(352,-188)|a|/=>/<495,0>[{\phantom{O}}`{\phantom{O}};{\nu}]
}
\defdiag{psiisadescentdatum}{   
  \morphism(1125,-450)|a|/->/<375,-450>[{\laxdea}`{\AAAA{(\mathsf{1})}};{\dd^\AAAA}]
  \morphism(1125,-450)|m|/->/<-375,-450>[{\laxdea}`{\AAAA{(\mathsf{1})}};{\dd^\AAAA}]
  \morphism(375,-450)|a|/<-/<750,0>[{\AAAA{(\mathsf{1})}}`{\laxdea};{\dd^\AAAA}]
  \morphism(375,-450)|l|/->/<-375,-450>[{\AAAA{(\mathsf{1})}}`{\AAAA{(\mathsf{2})}};{\AAAA(d^1)}]
  \morphism(1500,-900)|b|/->/<-375,-450>[{\AAAA{(\mathsf{1})}}`{\AAAA{(\mathsf{2})}};{\AAAA(d^0)}]
  \morphism(750,-900)|m|/->/<375,-450>[{\AAAA{(\mathsf{1})}}`{\AAAA{(\mathsf{2})}};{\AAAA(d^1)}]
  \morphism(0,-900)|m|/<-/<750,0>[{\AAAA{(\mathsf{2})}}`{\AAAA{(\mathsf{1})}};{\AAAA(d^0)}]
  \morphism(0,-900)|l|/->/<375,-450>[{\AAAA{(\mathsf{2})}}`{\AAAA{(\mathsf{3})}};{\AAAA(D^2)}]
  \morphism(375,-1350)|b|/<-/<750,0>[{\AAAA{(\mathsf{3})}}`{\AAAA{(\mathsf{2})}};{\AAAA(D^0)}]
  \morphism(1500,0)|a|/->/<-750,0>[{\s}`{\b};{H}]
  \morphism(1500,0)|r|/->/<-375,-450>[{\s}`{\laxdea};{J}]
  \morphism(750,0)|l|/->/<-375,-450>[{\b}`{\AAAA{(\mathsf{1})}};{\ran_H\left(\dd^\AAAA{J}\right)}]
  \morphism(862,-900)|a|/=>/<525,0>[{\phantom{O}}`{\phantom{O}};{\uppsi}]
  \morphism(300,-675)|a|/=>/<525,0>[{\phantom{O}}`{\phantom{O}};{\uppsi}]
  \morphism(300,-1125)|a|/=>/<525,0>[{\phantom{O}}`{\phantom{O}};{\AAAA(\sigma_{02})}]
  \morphism(675,-225)|a|/=>/<525,0>[{\phantom{O}}`{\phantom{O}};{\nu}]
}
\defdiag{consequence_of_applying_definition_of_varphi_added_step}{   
  \morphism(1875,-300)|a|/->/<750,-450>[{\laxdea}`{\AAAA{(\mathsf{1})}};{\dd^\AAAA}]
  \morphism(1875,-300)|m|/->/<-750,-450>[{\laxdea}`{\AAAA{(\mathsf{1})}};{\dd^\AAAA}]
  \morphism(1125,-750)|m|/<-/<0,450>[{\AAAA{(\mathsf{1})}}`{\b};{\ran_H\left(\dd^\AAAA{J}\right)}]
  \morphism(375,-300)|l|/->/<-375,-450>[{\AAAA{(\mathsf{1})}}`{\AAAA{(\mathsf{2})}};{\AAAA(d^1)}]
  \morphism(1125,-300)/=/<0,300>[{\b}`{\b};]
  \morphism(2625,-750)|b|/->/<-750,-450>[{\AAAA{(\mathsf{1})}}`{\AAAA{(\mathsf{2})}};{\AAAA(d^0)}]
  \morphism(1125,-750)|m|/->/<750,-450>[{\AAAA{(\mathsf{1})}}`{\AAAA{(\mathsf{2})}};{\AAAA(d^1)}]
  \morphism(0,-750)|m|/<-/<1125,0>[{\AAAA{(\mathsf{2})}}`{\AAAA{(\mathsf{1})}};{\AAAA(d^0)}]
  \morphism(0,-750)|l|/->/<375,-450>[{\AAAA{(\mathsf{2})}}`{\AAAA{(\mathsf{3})}};{\AAAA(D^2)}]
  \morphism(375,-1200)|b|/<-/<1500,0>[{\AAAA{(\mathsf{3})}}`{\AAAA{(\mathsf{2})}};{\AAAA(D^0)}]
  \morphism(1875,0)|a|/->/<-750,0>[{\s}`{\b};{H}]
  \morphism(1875,0)|r|/->/<0,-300>[{\s}`{\laxdea};{J}]
  \morphism(1125,-300)|l|/->/<-750,0>[{\b}`{\AAAA{(\mathsf{1})}};{\ran_H\left(\dd^\AAAA{J}\right)}]
  \morphism(1612,-750)|a|/=>/<525,0>[{\phantom{O}}`{\phantom{O}};{\uppsi}]
  \morphism(300,-525)|a|/=>/<525,0>[{\phantom{O}}`{\phantom{O}};{\varphi}]
  \morphism(375,-975)|a|/=>/<750,0>[{\phantom{O}}`{\phantom{O}};{\AAAA(\sigma_{02})}]
  \morphism(1238,-150)|a|/=>/<525,0>[{\phantom{O}}`{\phantom{O}};{\nu}]
}
\defdiag{consequence_of_applying_definition_of_varphi}{   
  \morphism(1350,0)|m|/->/<450,-450>[{\b}`{\AAAA{(\mathsf{1})}};{\ran_H\left(\dd^\AAAA{J}\right)}]
  \morphism(1350,0)|m|/->/<-450,-450>[{\b}`{\AAAA{(\mathsf{1})}};{\ran_H\left(\dd^\AAAA{J}\right)}]
  \morphism(450,0)|a|/<-/<900,0>[{\AAAA{(\mathsf{1})}}`{\b};{\ran_H\left(\dd^\AAAA{J}\right)}]
  \morphism(450,0)|l|/->/<-450,-450>[{\AAAA{(\mathsf{1})}}`{\AAAA{(\mathsf{2})}};{\AAAA(d^1)}]
  \morphism(1800,-450)|b|/->/<-450,-450>[{\AAAA{(\mathsf{1})}}`{\AAAA{(\mathsf{2})}};{\AAAA(d^0)}]
  \morphism(900,-450)|m|/->/<450,-450>[{\AAAA{(\mathsf{1})}}`{\AAAA{(\mathsf{2})}};{\AAAA(d^1)}]
  \morphism(0,-450)|m|/<-/<900,0>[{\AAAA{(\mathsf{2})}}`{\AAAA{(\mathsf{1})}};{\AAAA(d^0)}]
  \morphism(0,-450)|l|/->/<450,-450>[{\AAAA{(\mathsf{2})}}`{\AAAA{(\mathsf{3})}};{\AAAA(D^2)}]
  \morphism(450,-900)|b|/<-/<900,0>[{\AAAA{(\mathsf{3})}}`{\AAAA{(\mathsf{2})}};{\AAAA(D^0)}]
  \morphism(2250,0)|r|/->/<450,-450>[{\s}`{\laxdea};{J}]
  \morphism(2700,-450)|b|/->/<-900,0>[{\laxdea}`{\AAAA{(\mathsf{1})}};{\dd^\AAAA}]
  \morphism(2250,0)|l|/->/<-900,0>[{\s}`{\b};{H}]
  \morphism(1088,-450)|a|/=>/<525,0>[{\phantom{O}}`{\phantom{O}};{\varphi}]
  \morphism(412,-225)|a|/=>/<525,0>[{\phantom{O}}`{\phantom{O}};{\varphi}]
  \morphism(412,-675)|a|/=>/<525,0>[{\phantom{O}}`{\phantom{O}};{\AAAA(\sigma_{02})}]
  \morphism(1762,-225)|a|/=>/<525,0>[{\phantom{O}}`{\phantom{O}};{\nu}]
}
\defdiag{preservedkanextensiondzeroDzero}{   
  \morphism(0,0)|b|/->/<600,-600>[{\b}`{\AAAA(\mathsf{1})};{\ran_H\left(\dd{^\AAAA}{J}\right)}]
  \morphism(600,0)|a|/->/<-600,0>[{\s}`{\b};{H}]
  \morphism(600,0)|r|/->/<0,-300>[{\s}`{{\qquad\qquad\laxdea}};{J}]
  \morphism(600,-300)|r|/->/<0,-300>[{{\qquad\qquad\laxdea}}`{\AAAA(\mathsf{1})};{\dd{^\AAAA}}]
  \morphism(600,-600)|r|/->/<0,-300>[{\AAAA(\mathsf{1})}`{\AAAA{(\mathsf{2})}};{\AAAA(d^0)}]
  \morphism(600,-900)|r|/->/<0,-300>[{\AAAA{(\mathsf{2})}}`{\AAAA{(\mathsf{3})}};{\AAAA(D^0)}]
  \morphism(210,-300)|a|/{@{=>}@<13pt>}/<375,0>[{\phantom{O}}`{\phantom{O}};{\nu}]
}
\defdiag{leftsidevarphiidentity}{   
  \morphism(0,0)|a|/->/<525,0>[{\s}`{\laxdea};{J}]
  \morphism(0,0)|l|/->/<0,-375>[{\s}`{\b};{H}]
  \morphism(525,0)|r|/->/<0,-375>[{\laxdea}`{\AAAA{(\mathsf{1})}};{\dd{^\AAAA}}]
  \morphism(0,-375)|b|/->/<525,0>[{\b}`{\AAAA{(\mathsf{1})}};{\ran_H\left(\dd{^\AAAA}{J}\right)}]
  \morphism(0,-375)|l|/->/<0,-525>[{\b}`{\AAAA{(\mathsf{1})}};{\ran_H\left(\dd{^\AAAA}{J}\right)}]
  \morphism(525,-375)|m|/->/<0,-525>[{\AAAA{(\mathsf{1})}}`{\AAAA{(\mathsf{2})}};{\AAAA(d^0)}]
  \morphism(525,-375)/{@{=}@/^22pt/}/<300,-900>[{\AAAA{(\mathsf{1})}}`{\AAAA{(\mathsf{1})}};]
  \morphism(0,-900)|a|/->/<525,0>[{\AAAA{(\mathsf{1})}}`{\AAAA{(\mathsf{2})}};{\AAAA(d^1)}]
  \morphism(0,-900)/{@{=}@/_20pt/}/<825,-375>[{\AAAA{(\mathsf{1})}}`{\AAAA{(\mathsf{1})}};]
  \morphism(525,-900)|m|/->/<300,-375>[{\AAAA{(\mathsf{2})}}`{\AAAA{(\mathsf{1})}};{\AAAA(s^0)}]
  \morphism(150,-1088)|a|/{@{=>}@<-3pt>}/<450,0>[{\phantom{O}}`{\phantom{O}};{\AAAA(\mathfrak{n}_1)^{-1}}]
  \morphism(540,-825)|a|/=>/<345,0>[{\phantom{O}}`{\phantom{O}};{\AAAA(\mathfrak{n}_0)}]
  \morphism(75,-638)|a|/=>/<375,0>[{\phantom{O}}`{\phantom{O}};{\varphi}]
  \morphism(75,-188)|a|/=>/<375,0>[{\phantom{O}}`{\phantom{O}};{\nu}]
}
\defdiag{rightsidevarphiidentity}{   
  \morphism(0,0)|a|/->/<0,-525>[{\s}`{\b};{H}]
  \morphism(0,0)|a|/->/<525,0>[{\s}`{\laxdea};{J}]
  \morphism(0,-525)|r|/->/<525,0>[{\b}`{\AAAA{(\mathsf{1})}};{\ran_H\left(\dd{^\AAAA}{J}\right)}]
  \morphism(525,0)|a|/->/<525,0>[{\laxdea}`{\AAAA{(\mathsf{1})}};{\dd{^\AAAA}}]
  \morphism(525,0)|m|/->/<0,-525>[{\laxdea}`{\AAAA{(\mathsf{1})}};{\dd{^\AAAA}}]
  \morphism(1050,0)|m|/->/<0,-525>[{\AAAA{(\mathsf{1})}}`{\AAAA{(\mathsf{2})}};{\AAAA(d^0)}]
  \morphism(1050,0)/{@{=}@/^22pt/}/<300,-900>[{\AAAA{(\mathsf{1})}}`{\AAAA{(\mathsf{1})}};]
  \morphism(525,-525)|a|/->/<525,0>[{\AAAA{(\mathsf{1})}}`{\AAAA{(\mathsf{2})}};{\AAAA(d^1)}]
  \morphism(525,-525)/{@{=}@/_20pt/}/<825,-375>[{\AAAA{(\mathsf{1})}}`{\AAAA{(\mathsf{1})}};]
  \morphism(1050,-525)|m|/->/<300,-375>[{\AAAA{(\mathsf{2})}}`{\AAAA{(\mathsf{1})}};{\AAAA(s^0)}]
  \morphism(675,-712)|a|/{@{=>}@<-3pt>}/<450,0>[{\phantom{O}}`{\phantom{O}};{\AAAA(\mathfrak{n}_1)^{-1}}]
  \morphism(1065,-450)|a|/=>/<345,0>[{\phantom{O}}`{\phantom{O}};{\AAAA(\mathfrak{n}_0)}]
  \morphism(600,-262)|a|/=>/<375,0>[{\phantom{O}}`{\phantom{O}};{\uppsi}]
  \morphism(75,-262)|a|/=>/<375,0>[{\phantom{O}}`{\phantom{O}};{\nu}]
}
\defdiag{equation_on_the_descent_datum_leftside_KanExtension}{   
  \morphism(450,0)|r|/->/<0,-300>[{\s}`{\laxdea};{\check{J}}]
  \morphism(450,-300)|l|/->/<-450,-300>[{\laxdea}`{\AAAA{(\mathsf{1})}};{\dd^{\AAAA}}]
  \morphism(450,-300)|r|/->/<450,-300>[{\laxdea}`{\AAAA{(\mathsf{1})}};{\dd^{\AAAA}}]
  \morphism(0,-600)|l|/->/<450,-300>[{\AAAA{(\mathsf{1})}}`{\AAAA{(\mathsf{2})}};{\AAAA(d^1)}]
  \morphism(900,-600)|r|/->/<-450,-300>[{\AAAA{(\mathsf{1})}}`{\AAAA{(\mathsf{2})}};{\AAAA(d^0)}]
  \morphism(188,-600)|a|/=>/<525,0>[{\phantom{O}}`{\phantom{O}};{\uppsi}]
}
\defdiag{equation_on_the_descent_datum_rightside_KanExtension}{   
  \morphism(450,0)|l|/->/<-450,-600>[{\s}`{\AAAA{(\mathsf{1})}};{J}]
  \morphism(450,0)|r|/->/<450,-600>[{\s}`{\AAAA{(\mathsf{1})}};{J}]
  \morphism(0,-600)|l|/->/<450,-300>[{\AAAA{(\mathsf{1})}}`{\AAAA{(\mathsf{2})}};{\AAAA(d^1)}]
  \morphism(900,-600)|r|/->/<-450,-300>[{\AAAA{(\mathsf{1})}}`{\AAAA{(\mathsf{2})}};{\AAAA(d^0)}]
  \morphism(188,-450)|a|/=>/<525,0>[{\phantom{O}}`{\phantom{O}};{\varphi}]
}
\defdiag{omegaforthekanextension}{   
  \morphism(0,0)|l|/->/<600,-600>[{\b}`{\laxdea};{R}]
  \morphism(600,0)|a|/->/<-600,0>[{\s}`{\b};{H}]
  \morphism(600,0)|r|/->/<0,-600>[{\s}`{\laxdea};{J}]
  \morphism(210,-300)|a|/{@{=>}@<13pt>}/<405,0>[{\phantom{O}}`{\phantom{O}};{\omega}]
}
\defdiag{equationKanextensioninordertodefinebetaomega}{   
  \morphism(0,0)|b|/->/<600,-300>[{\b}`{\laxdea};{R}]
  \morphism(600,0)|a|/->/<-600,0>[{\s}`{\b};{H}]
  \morphism(600,0)|r|/->/<0,-300>[{\s}`{\laxdea};{J}]
  \morphism(600,-300)|l|/->/<0,-300>[{\laxdea}`{\AAAA{(\mathsf{1})}};{\dd^\AAAA}]
  \morphism(300,-150)|a|/{@{=>}@<5pt>}/<300,0>[{\phantom{O}}`{\phantom{O}};{\omega}]
}
\defdiag{equationKanextensioninordertodefinebetauniversalproperty}{   
  \morphism(0,0)|m|/->/<600,-600>[{\b}`{\AAAA{(\mathsf{1})}};{\ran_H\left(\dd{^\AAAA}{J}\right)}]
  \morphism(600,0)|a|/->/<-600,0>[{\s}`{\b};{H}]
  \morphism(600,0)|r|/->/<0,-600>[{\s}`{\AAAA{(\mathsf{1})}};{\dd^\AAAA\circ\,{J}}]
  \morphism(0,0)|l|/->/<0,-600>[{\b}`{\laxdea};{R}]
  \morphism(0,-600)|b|/->/<600,0>[{\laxdea}`{\AAAA{(\mathsf{1})}};{\dd^\AAAA}]
  \morphism(262,-300)|a|/{@{=>}@<12pt>}/<375,0>[{\phantom{O}}`{\phantom{O}};{\nu}]
  \morphism(0,-300)|a|/{@{=>}@<-14pt>}/<375,0>[{\phantom{O}}`{\phantom{O}};{\beta}]
}
\defdiag{equationKanextensioninordertodefinebetauniversalpropertyvvlinha}{   
  \morphism(0,0)|m|/->/<600,-300>[{\b}`{{\laxdea}};{\check{J}}]
  \morphism(600,0)|a|/->/<-600,0>[{\s}`{\b};{H}]
  \morphism(600,0)|r|/->/<0,-300>[{\s}`{{\laxdea}};{{J}}]
  \morphism(600,-300)|r|/->/<0,-300>[{{\laxdea}}`{\AAAA{(\mathsf{1})}};{\dd^\AAAA}]
  \morphism(0,0)|l|/->/<0,-600>[{\b}`{\laxdea};{R}]
  \morphism(0,-600)|b|/->/<600,0>[{\laxdea}`{\AAAA{(\mathsf{1})}};{\dd^\AAAA}]
  \morphism(300,-150)|a|/{@{=>}@<5pt>}/<300,0>[{\phantom{O}}`{\phantom{O}};{\tilde{\upnu}}]
  \morphism(52,-300)|a|/{@{=>}@<-10pt>}/<375,0>[{\phantom{O}}`{\phantom{O}};{\beta}]
}
\defdiag{deAomegaintermsofbetanutildeuppsirightside}{   
  \morphism(0,0)|b|/->/<600,-300>[{\b}`{\laxdea};{R}]
  \morphism(600,0)|a|/->/<-600,0>[{\s}`{\b};{H}]
  \morphism(600,0)|r|/->/<0,-300>[{\s}`{\laxdea};{J}]
  \morphism(600,-300)|l|/->/<0,-300>[{\laxdea}`{\AAAA{(\mathsf{1})}};{\dd^\AAAA}]
  \morphism(600,-300)|a|/->/<600,0>[{\laxdea}`{\AAAA{(\mathsf{1})}};{\dd^\AAAA}]
  \morphism(1200,-300)|r|/->/<0,-300>[{\AAAA{(\mathsf{1})}}`{\AAAA{(\mathsf{2})}};{\AAAA(d^0)}]
  \morphism(600,-600)|b|/->/<600,0>[{\AAAA{(\mathsf{1})}}`{\AAAA{(\mathsf{2})}};{\AAAA(d^1)}]
  \morphism(300,-150)|a|/{@{=>}@<5pt>}/<300,0>[{\phantom{O}}`{\phantom{O}};{\omega}]
  \morphism(638,-450)|a|/=>/<525,0>[{\phantom{O}}`{\phantom{O}};{\uppsi}]
}
\defdiag{deAomegaintermsofbetanutildeuppsi}{   
  \morphism(0,0)|m|/->/<600,-300>[{\b}`{{\laxdea}};{\check{J}}]
  \morphism(600,0)|a|/->/<-600,0>[{\s}`{\b};{H}]
  \morphism(600,0)|r|/->/<0,-300>[{\s}`{{\laxdea}};{{J}}]
  \morphism(600,-300)|l|/->/<0,-300>[{{\laxdea}}`{\AAAA{(\mathsf{1})}};{\dd^\AAAA}]
  \morphism(0,0)|l|/->/<0,-600>[{\b}`{\laxdea};{R}]
  \morphism(0,-600)|b|/->/<600,0>[{\laxdea}`{\AAAA{(\mathsf{1})}};{\dd^\AAAA}]
  \morphism(600,-300)|a|/->/<600,0>[{{\laxdea}}`{\AAAA{(\mathsf{1})}};{\dd^\AAAA}]
  \morphism(1200,-300)|r|/->/<0,-300>[{\AAAA{(\mathsf{1})}}`{\AAAA{(\mathsf{2})}};{\AAAA(d^0)}]
  \morphism(600,-600)|b|/->/<600,0>[{\AAAA{(\mathsf{1})}}`{\AAAA{(\mathsf{2})}};{\AAAA(d^1)}]
  \morphism(300,-150)|a|/{@{=>}@<5pt>}/<300,0>[{\phantom{O}}`{\phantom{O}};{\tilde{\upnu}}]
  \morphism(52,-300)|a|/{@{=>}@<-10pt>}/<375,0>[{\phantom{O}}`{\phantom{O}};{\beta}]
  \morphism(638,-450)|a|/=>/<525,0>[{\phantom{O}}`{\phantom{O}};{\uppsi}]
}
\defdiag{rightsidetoseethatbetaisatwocell}{   
  \morphism(0,0)|m|/->/<600,-300>[{\b}`{{\laxdea}};{\check{J}}]
  \morphism(600,0)|a|/->/<-600,0>[{\s}`{\b};{H}]
  \morphism(600,0)|r|/->/<0,-300>[{\s}`{{\laxdea}};{{J}}]
  \morphism(600,-300)|r|/->/<0,-300>[{{\laxdea}}`{\AAAA{(\mathsf{1})}};{\dd^\AAAA}]
  \morphism(0,0)|l|/->/<0,-600>[{\b}`{\laxdea};{R}]
  \morphism(0,-600)|a|/->/<600,0>[{\laxdea}`{\AAAA{(\mathsf{1})}};{\dd^\AAAA}]
  \morphism(0,-900)|b|/->/<600,0>[{\AAAA{(\mathsf{1})}}`{\AAAA{(\mathsf{2})}};{\AAAA(d^1)}]
  \morphism(600,-600)|r|/->/<0,-300>[{\AAAA{(\mathsf{1})}}`{\AAAA{(\mathsf{2})}};{\AAAA(d^0)}]
  \morphism(0,-600)|l|/->/<0,-300>[{\laxdea}`{\AAAA{(\mathsf{1})}};{\dd^\AAAA}]
  \morphism(300,-150)|a|/{@{=>}@<5pt>}/<300,0>[{\phantom{O}}`{\phantom{O}};{\tilde{\upnu}}]
  \morphism(52,-300)|a|/{@{=>}@<-10pt>}/<375,0>[{\phantom{O}}`{\phantom{O}};{\beta}]
  \morphism(38,-750)|a|/=>/<525,0>[{\phantom{O}}`{\phantom{O}};{\uppsi}]
}
\defdiag{leftsideoftheequationforbetalaxdescent}{   
  \morphism(0,0)|a|/->/<600,-300>[{\b}`{{\laxdea}};{\check{J}}]
  \morphism(600,-300)|l|/->/<0,-300>[{{\laxdea}}`{\AAAA{(\mathsf{1})}};{\dd^\AAAA}]
  \morphism(0,0)|l|/->/<0,-600>[{\b}`{\laxdea};{R}]
  \morphism(0,-600)|b|/->/<600,0>[{\laxdea}`{\AAAA{(\mathsf{1})}};{\dd^\AAAA}]
  \morphism(600,-300)|a|/->/<600,0>[{{\laxdea}}`{\AAAA{(\mathsf{1})}};{\dd^\AAAA}]
  \morphism(1200,-300)|r|/->/<0,-300>[{\AAAA{(\mathsf{1})}}`{\AAAA{(\mathsf{2})}};{\AAAA(d^0)}]
  \morphism(600,-600)|b|/->/<600,0>[{\AAAA{(\mathsf{1})}}`{\AAAA{(\mathsf{2})}};{\AAAA(d^1)}]
  \morphism(52,-300)|a|/{@{=>}@<-10pt>}/<375,0>[{\phantom{O}}`{\phantom{O}};{\beta}]
  \morphism(638,-450)|a|/=>/<525,0>[{\phantom{O}}`{\phantom{O}};{\uppsi}]
}
\defdiag{rightsideoftheequationforbetalaxdescent}{   
  \morphism(0,0)|a|/->/<600,-300>[{\b}`{{\laxdea}};{\check{J}}]
  \morphism(600,-300)|r|/->/<0,-300>[{{\laxdea}}`{\AAAA{(\mathsf{1})}};{\dd^\AAAA}]
  \morphism(0,0)|l|/->/<0,-600>[{\b}`{\laxdea};{R}]
  \morphism(0,-600)|a|/->/<600,0>[{\laxdea}`{\AAAA{(\mathsf{1})}};{\dd^\AAAA}]
  \morphism(0,-900)|b|/->/<600,0>[{\AAAA{(\mathsf{1})}}`{\AAAA{(\mathsf{2})}};{\AAAA(d^1)}]
  \morphism(600,-600)|r|/->/<0,-300>[{\AAAA{(\mathsf{1})}}`{\AAAA{(\mathsf{2})}};{\AAAA(d^0)}]
  \morphism(0,-600)|l|/->/<0,-300>[{\laxdea}`{\AAAA{(\mathsf{1})}};{\dd^\AAAA}]
  \morphism(52,-300)|a|/{@{=>}@<-10pt>}/<375,0>[{\phantom{O}}`{\phantom{O}};{\beta}]
  \morphism(38,-750)|a|/=>/<525,0>[{\phantom{O}}`{\phantom{O}};{\uppsi}]
}
\defdiag{omegaforthekanextensionlinha}{   
  \morphism(0,0)|l|/->/<600,-600>[{\b}`{\laxdea};{R}]
  \morphism(600,0)|a|/->/<-600,0>[{\s}`{\b};{H}]
  \morphism(600,0)|r|/->/<0,-600>[{\s}`{\laxdea};{J}]
  \morphism(210,-300)|a|/{@{=>}@<13pt>}/<405,0>[{\phantom{O}}`{\phantom{O}};{\omega}]
}
\defdiag{kan_extension_rightside_prova_nocasodenu}{   
  \morphism(0,0)|m|/->/<600,-600>[{\b}`{\laxdea};{\check{J}}]
  \morphism(0,0)|l|/{@{->}@/_30pt/}/<600,-600>[{\b}`{\laxdea};{R}]
  \morphism(600,0)|a|/->/<-600,0>[{\s}`{\b};{H}]
  \morphism(600,0)|r|/->/<0,-600>[{\s}`{\laxdea};{J}]
  \morphism(210,-300)|a|/{@{=>}@<18pt>}/<375,0>[{\phantom{O}}`{\phantom{O}};{\tilde{\upnu}}]
  \morphism(75,-300)|a|/{@{=>}@<-18pt>}/<375,0>[{\phantom{O}}`{\phantom{O}};{\tilde{\upbeta}}]
}
\defdiag{eqpinternalgroupoid}{   
  \morphism(0,0)/->/<1200,0>[{e}`{e\times_be};]
  \morphism(1200,0)|a|/{@{->}@/_20pt/}/<-1200,0>[{e\times_be}`{e};{\pi^e}]
  \morphism(1200,0)|b|/{@{->}@/^20pt/}/<-1200,0>[{e\times_be}`{e};{\pi_e}]
  \morphism(2400,0)/->/<-1200,0>[{e\times_be\times_be}`{e\times_be};]
  \morphism(2400,0)/{@{->}@/_20pt/}/<-1200,0>[{e\times_be\times_be}`{e\times_be};]
  \morphism(2400,0)/{@{->}@/^20pt/}/<-1200,0>[{e\times_be\times_be}`{e\times_be};]
}
\defdiag{equation_on_the_descent_datum_leftsideforthecaseofpseudofunctor}{   
  \morphism(600,0)|r|/->/<0,-375>[{\FFFF(b)}`{\laxde\left(\FFFF^p\right)};{K_{p}}]
  \morphism(600,-375)|l|/->/<-600,-450>[{\laxde\left(\FFFF^p\right)}`{\FFFF^p(\mathsf{1})=\FFFF(e)};{\dd^{\FFFF^p}}]
  \morphism(600,-375)|r|/->/<600,-450>[{\laxde\left(\FFFF^p\right)}`{\FFFF(e)=\FFFF^p(\mathsf{1})};{\dd^{\FFFF^p}}]
  \morphism(0,-825)|l|/->/<600,-450>[{\FFFF^p(\mathsf{1})=\FFFF(e)}`{\FFFF(e\times_be)=\FFFF^p(\mathsf{2})};{\FFFF^p(d^1)=\FFFF(\pi_e)}]
  \morphism(1200,-825)|r|/->/<-600,-450>[{\FFFF(e)=\FFFF^p(\mathsf{1})}`{\FFFF(e\times_be)=\FFFF^p(\mathsf{2})};{\FFFF(\pi^e)=\FFFF^p(d^0)}]
  \morphism(412,-825)|a|/=>/<375,0>[{\phantom{O}}`{\phantom{O}};{\uppsi}]
}
\defdiag{equation_on_the_descent_datum_rightsideforthecaseofthepseudofunctor}{   
  \morphism(525,0)|l|/->/<-525,-825>[{\FFFF(b)}`{\FFFF(e)};{\FFFF(p)}]
  \morphism(525,0)|r|/->/<525,-825>[{\FFFF(b)}`{\FFFF(e)};{\FFFF(p)}]
  \morphism(0,-825)|l|/->/<525,-450>[{\FFFF(e)}`{\FFFF(e\times_be)};{\FFFF(\pi_e)}]
  \morphism(1050,-825)|r|/->/<-525,-450>[{\FFFF(e)}`{\FFFF(e\times_be)};{\FFFF(\pi^e)}]
  \morphism(525,0)|m|/->/<0,-1275>[{\FFFF(b)}`{\FFFF(e\times_be)};{\FFFF(\pi_e\cdot{p})=\FFFF(\pi^e\cdot{p})}]
  \morphism(562,-638)|a|/{@{=>}@<-23pt>}/<375,0>[{\phantom{O}}`{\phantom{O}};{\ffff_{{}_{\pi^e\,{p}}}^{-1}}]
  \morphism(112,-638)|a|/{@{=>}@<-23pt>}/<375,0>[{\phantom{O}}`{\phantom{O}};{\ffff_{{}_{\pi_e\,{p}}}}]
}

\def\pu{}
\fi


\title{Descent data and absolute Kan extensions} 
\author{Fernando Lucatelli Nunes}
\address{Department of Information and Computing Sciences, Utrecht University}
\eaddress{f.lucatellinunes@uu.nl}
\amsclass{18N10, 18C15, 18C20, 18F20, 18A22, 18A30, 18A40}


\thanks{This research was partially supported by the Institut de Recherche en Math\'{e}matique et Physique (IRMP, UCLouvain, Belgium), and by the Centre for Mathematics of the University of Coimbra - UIDB/00324/2020, funded by the Portuguese Government through FCT/MCTES.}

\maketitle 

\begin{abstract}
The fundamental construction underlying descent theory, 
the lax descent category, comes with a functor that forgets the 
\textit{descent data}. We prove that,
in any $2$-category $\A $ with lax descent objects, the
forgetful morphisms
 create all Kan extensions that are preserved by certain morphisms.  As  a consequence, in the case $\A = \Cat $,
  we get a \textit{monadicity theorem} which says that a right adjoint functor is monadic if it is, up  to the composition with an equivalence, (naturally isomorphic to) a functor that forgets descent data. In particular, within the classical context    of \textit{descent theory}, we show that, in a fibred category, 
  the forgetful functor between the category of internal actions of a precategory $a$
  and the category of internal actions of the underlying discrete precategory is monadic if and only if it has a left adjoint. More particularly, this shows that 
  one of the
implications of the celebrated B\'{e}nabou-Roubaud theorem 
does not depend on the 
so called Beck-Chevalley condition. Namely, we prove that, in indexed categories, 
whenever an effective descent morphism induces a right adjoint functor, the induced functor is  monadic.
\end{abstract}

\tableofcontents  
\setcounter{secnumdepth}{-1}
\section{Introduction}

The various notions of \textit{descent objects}, the $2$-dimensional limits underlying \textit{descent theory}, can be seen as $2$-dimensional analogues
of the equalizer. While equalizers encompass equality and
commutativity of diagrams in $1$-dimensional category theory,
the  (lax) descent objects encompass $2$-dimensional 
coherence: morphism (or $2$-cell) plus
coherence equations. 

For this reason, results on the (lax) descent objects 
 usually shed light on a wide
range of situations, including, of course, Grothendieck descent theory (\textit{e.g.}~\cite{MR1466540, MR2107402, 2016arXiv160604999L}), Janelidze-Galois theory~\cite{MR2073649}, algebraic topology~\cite{MR1453515}, two-dimensional monad theory (\textit{e.g.}~\cite{MR1935980, MR3491845}), and two-dimensional category theory (\textit{e.g.}~\cite{2019arXiv190201225L}).

As shown in \cite{MR1466540}, in the classical case of the $2$-category 
$\Cat $ of categories,
 \textit{internal category theory}
provides a useful perspective to introduce 
\textit{descent theory} or, more particularly, the lax descent category.
The lax descent category can be seen as a generalization of the $2$-functor 
$$\mathsf{Mon}(\Set)^\op \to \Cat, \quad m\mapsto m\textrm{-}\Set $$
in which $\mathsf{Mon}(\Set)$ denotes the usual category of monoids (of the cartesian monoidal category $\Set $),
and $m\textrm{-}\Set$ is the category of sets endowed with actions of the monoid $m$, usually called
$m$-sets.

Recall that every small category $a $ (internal category in $\Set $) has an underlying
truncated simplicial set, called the underlying \textit{precategory} 
$$\Catt (\j - , a)  : \Delta _ {\mathrm{3}}^\op\to \Set $$
%
\pu
\begin{equation*}
\diag{underlyingprecategory}
\end{equation*}
in which, denoting by $\Delta $ the category
 of the finite non-empty ordinals and order preserving
functions,
 $\j : \Delta _ {\mathrm{3}}\to \Cat $ is the usual inclusion
given by the composition of the inclusions 
$\Delta _ {\mathrm{3}}\to \Delta \to \Cat $.

It is well known that there is a fully faithful functor
$\Sigma : \mathsf{Mon}(\Set)\to \Cat (\Set ) $ between the category of monoids (internal monoids in $\Set $)
and the category of small categories (internal categories in $\Set $) that
associates each monoid with the corresponding single object category. The underlying precategory of $\Sigma m $
is given by
$$\Sigma m : \Delta _ {\mathrm{3}}^\op\to \Set $$
%
\pu
\begin{equation*}
\diag{underlyingprecategoryofamonoid}
\end{equation*}
in which $m$ is the underlying set of the monoid,  $\left\{m \right\}$ is the singleton with $m$ as element, $\Sigma m(D_2), \Sigma m(D_0): m\times m \to m  $
are the two product projections, $\Sigma m(D_1)$ is the operation of the monoid, 
and $\Sigma m(s_0)$ gives the unit.
In this context, the objects and morphisms of the
category $m\textrm{-}\Set $ 
 can be described internally in $\Set $ as follows.

 Since $\Set $ has pullbacks, 
 we can consider the (basic) \textit{indexed category}, that is to say, the pseudofunctor
\begin{eqnarray*}
\Set /- : &\Set ^\op &\to \Cat\\ 
        &   w       & \mapsto \Set /w  \\
        &   f      &	\mapsto f^\ast
\end{eqnarray*}	
in which $ \Set /w $ denotes the comma category,
and 
$f^\ast $ denotes the \textit{change of base functor}
(given by the pullback along $f$).

An $m$-set is a set $w$ endowed with an endomorphism 
$\xi  $
of the projection $\textrm{proj}_{m} : m\times w\to m $ in the  comma category $\Set / m $, subject to
the equations
$$  \mathfrak{p}\cdot m(s_0)^\ast (\xi )\cdot \mathfrak{p} = \id _{\Set }, \quad 
m(D_0)^\ast (\xi )\cdot \mathfrak{p}\cdot  m(D_2)^\ast (\xi )
= \mathfrak{p}\cdot m(D_1)^\ast (\xi )\cdot \mathfrak{p}
$$ 
in which, by abuse of language, we denote by $\mathfrak{p}$ the appropriate 
\textit{canonical isomorphisms} given by the pseudofunctor $\Set /- $ 
(induced by the universal properties of the pullbacks in each case). 
These equations 
correspond to the identity and associativity equations
for the action.
The morphisms $(w, \xi)\to (w', \xi ') $ of $m$-sets are morphisms 
(functions) $w\to w'$
 between the underlying sets respecting the structures $\xi $ and $\xi ' $.

This viewpoint gives $m\textit{-}\Set $ precisely as the 
\textit{lax descent category} of the composition of 
$\op\left(\Sigma m\right) : \Delta _ {\mathrm{3}}\to \Set ^\op $  
with the pseudofunctor $\Set /- : \Set ^\op\to\Cat $.
More generally, given a small category $a $, 
the lax descent category (see Definition \ref{definitionlaxdescentcategory}) of
%
\pu
\begin{equation*}
	\diag{descentunderlyingprecategory}
\end{equation*}
is equivalent to the category  
$\Cat \left[ a , \Set \right] $
of functors $a\to \Set $ and natural transformations,
that is to say, the category of \textit{actions
of the small category $a$ in $\Set$}.

In order to reach the level of abstraction of \cite{MR1466540}, firstly it 
should be noted that the definitions above can be 
considered in any category $\mathbb{C} $
with pullbacks, using the basic
indexed category $\mathbb{C}/- : \mathbb{C}^\op\to\Cat $. That is to say,
we get the (basic) internal notion
of the category of actions 
$a\to \mathbb{C} $ for each internal category $a $.
Secondly, we can replace the pseudofunctor
$\mathbb{C}/-  $
by any other pseudofunctor (indexed category)
$\mathcal{F} : \mathbb{C}^\op\to \Cat $ 
 of interest. By definition, given an internal
 (pre)category $a : \Delta _ \mathrm{3} ^\op \to \mathbb{C} $ of $\mathbb{C}$, the lax descent
 category of
%
\pu
\begin{equation*}
\diag{internalactionsdescentunderlyingprecategory}
\end{equation*}
is the category of \textit{$\mathcal{F}$-internal actions}
of $a$ in $\mathbb{C}$.

Recall that, if $\mathbb{C}$ has pullbacks, given a 
morphism $p: e\to b $, the kernel pair induces a precategory
which is actually the underlying precategory of an \textit{internal groupoid} of $\mathbb{C} $, denoted herein by $\mathsf{Eq}(p) $. 
Following the definition,
 given any pseudofunctor $\FFFF : \mathbb{C} ^\op\to \Cat $,
we have that the category of $\FFFF$-internal actions of  $\mathsf{Eq}(p) $
is given by the lax descent category 
$\mathrm{lax}\textrm{-}\mathcal{D}\mathrm{esc} \left( \FFFF\cdot \op \left(\mathsf{Eq}(p)\right)\right)$. 
In this case, the universal property of the lax descent category induces a factorization 
\begin{equation*}\tag{$\FFFF$-descent factorization of $\FFFF (p) $}
\label{introductiondescentfactorization}
\xymatrix{
\FFFF (b)
\ar[rr]|-{\FFFF (p) }
\ar[rd]
&&
\FFFF (e)
\\
&
\mathrm{lax}\textrm{-}\mathcal{D}\mathrm{esc}  \left( \FFFF\cdot \op \left(\mathsf{Eq}(p)\right)\right)
\ar[ru]
&
}
\end{equation*}
in which $\mathrm{lax}\textrm{-}\mathcal{D}\mathrm{esc} \left( \FFFF\,  \mathsf{Eq}(p) \right)\to \FFFF (e) $ is the forgetful functor that forgets
descent data (see, for instance,  \cite[Section~3]{MR1466540} or, more appropriately to our context, Lemma \ref{factorizationdescentfacets} below). 

In this setting, B\'{e}nabou and Roubaud~\cite{MR0255631} showed that, if $\FFFF : \mathbb{C} ^\op\to \Cat $
comes from a bifibration satisfying the so called \textit{Beck-Chevalley condition} (see, for instance, \cite[Section~7]{2016arXiv160604999L} or Section~\ref{sectionfinale} below), then 
the \ref{introductiondescentfactorization}
is equivalent to the Eilenberg-Moore factorization of the adjunction $\FFFF (p)!\dashv \FFFF (p) $. In particular, in this case, 
$\FFFF (p) $ is monadic if and only if $p $ is of 
\textit{effective $\FFFF $-descent}
(which means that $\FFFF (b)\to \mathrm{lax}\textrm{-}\mathcal{D}\mathrm{esc} \left( \FFFF\cdot \op \left(\mathsf{Eq}(p)\right)\right)$ is an equivalence).

The main result of the present paper is 
 within the general context of 
the  lax descent object of a truncated
pseudocosimplicial object inside a $2$-category $\A $. More precisely, our main theorem says that, for any given truncated pseudocosimplicial
object $$\AAAA : \Delta _\mathrm{3}\to \A $$ 
%
\pu
\begin{equation*}
	\diag{pseudofunctorfromdeltatres}
\end{equation*}
the \textit{forgetful morphism}  $ \dd ^\AAAA : \laxdea \to \AAAA   $  creates
 the right Kan extensions that are preserved  
by   $\AAAA (d^0)$ and
$\AAAA (D^0)\cdot \AAAA (d^0)$.
In particular, such forgetful morphisms create \textit{absolute Kan extensions}. If 
$\A = \Cat $, we get in particular that the functor  $ \dd ^\AAAA : \laxdea \to \AAAA   $
that forgets descent data creates
 absolute limits and colimits.

The main theorem implies that,
given any pseudofunctor $\FFFF : \mathbb{C} ^\op
\to\Cat $,  the 
forgetful functor 
$$\mathrm{lax}\textrm{-}\mathcal{D}\mathrm{esc} \left( \FFFF\,\circ\, \op (a) \right)\to\mathcal{F} \, a(\mathsf{1}) $$
between the $\FFFF$-internal actions of a precategory 
$a : \Delta _ {\mathrm{3} } ^\op\to \mathbb{C} $
 and the category of internal actions of the \textit{underlying discrete precategory} of 
$a$ creates absolute limits and colimits. This generalizes the
fact that, if $a$ is actually a small category, 
the forgetful functor (restriction functor)  
$$\Cat \left[ a , \Set \right]\to  \Cat \left[ 
\overline{a (\mathsf{1})} , \Set \right] $$
creates absolute limits and colimits,
in which, by abuse of language, $\overline{a (\mathsf{1})}$
denotes the \textit{underlying discrete category} of $a$ (see, for instance, \cite[Proposition~2.21]{MR0470019}).

As a particular case of this conclusion, given any indexed category 
$\FFFF : \mathbb{C} ^\op\to \Cat ,$ whenever $p $ is of effective $\FFFF $-descent,
$\FFFF (p) $ creates absolute limits and colimits. 
Therefore, by Beck's monadicity theorem,
assuming that $\FFFF (p) $ has a left adjoint, 
if $p $ is of effective $\FFFF $-descent then 
$\FFFF (p) $ is monadic (Theorem \ref{MONADIC} and Theorem \ref{consequenceofthemaintheoreminthecaseofGrothendieckdescent}). 

This shows that, if  $\FFFF $ comes from a bifibration,
one of the 
implications of the B\'{e}nabou-Roubaud theorem 
\textit{does not depend} on the 
 Beck-Chevalley condition. Namely, \textit{in a bifibred category with pullbacks, 
effective descent morphisms always induce monadic functors}. 

It should be observed that it is known that,  without assuming the Beck-Chevalley condition,
monadicity of $\FFFF (p) $ does not imply that
$p$ is of effective $\FFFF$-descent. This is shown
for instance in \cite[Remark~7]{MR2107401},
where Sobral, considering the indexed category $\cat^\op\to\Cat $
of op-fibrations in the category of small categories,
provides an example of a morphism
that is not of effective descent but does induce 
a monadic functor.

In Section \ref{Sectionofthelaxdescentcategory}, we briefly give the basic definition of the lax descent category, and present the corresponding definition for a general $2$-category. Namely,  a $2$-dimensional limit called \textit{the lax descent object} (see \cite[Section~5]{MR0401868} or, for pseudofunctors, \cite[Section~3]{2016arXiv160703087L}). We
mostly follow the approach of \cite[Section~2]{2019arXiv190201225L} except for starting with
pseudofunctors $(\AAAA , \aaaa): \Delta _{\mathrm{3}}\to\A $ instead of using 
a strict replacement of the domain.

In Section \ref{Sectionforgetfulmorphismsandkanextensions}, we establish our main theorems 
on 
the \textit{morphisms that forget descent data}. In order to do so, we start by recalling the definitions 
on Kan extensions  inside a $2$-category (\textit{e.g.} \cite[Section~2]{MR0463261}). Then, we prove Theorem \ref{maintheorem} and show the main consequences,
including a monadicity theorem (Theorem \ref{MONADIC}). 
We also show how Theorem \ref{MONADIC}
and the monadicity theorem of \cite[Section~5]{2019arXiv190201225L} implies in a new monadicity characterization in Remark \ref{monadicityinCat} (Theorem \ref{monadicityinCat2}). It says that  \textit{a right adjoint functor is monadic if, and only if, it is a functor that forgets descent data composed with an equivalence}.

Section \ref{Sectiondescenttheory} establishes the setting of \textit{Grothendieck descent theory}~\cite{MR1285884, 2016arXiv160604999L}. The main aim of the section is to establish Lemma \ref{factorizationdescentfacets} in order to recover 
the usual \textit{descent factorization} (see, for instance, \cite[Section~3]{MR1466540})  directly 
via the universal property of \textit{lax descent category}.

In Section \ref{sectionfinale}, we give 
aspects of the relation between monadicity and effective descent morphisms in the context of \cite{MR0255631,  MR1285884, 2016arXiv160604999L}. We recall the \textit{Beck-Chevalley condition} and the B\'{e}nabou-Roubaud theorem. We discuss examples of non-effective descent morphisms
inducing monadic functors. Finally, we also establish and discuss  the main consequences of our 
Theorem \ref{maintheorem} in this context, including the result that, \textit{in bifibred categories, effective descent morphisms always induce monadic functors}, even without satisfying the Beck-Chevalley condition.

\setcounter{secnumdepth}{5}

\section{The lax descent category}\label{Sectionofthelaxdescentcategory}
Let $\Cat $ be the cartesian closed category of categories in some universe. We denote the \textit{internal hom}  by 
$$\Cat[-,-]: \Cat ^\op\times \Cat\to \Cat ,$$
which of course is a $2$-functor ($\Cat $-enriched functor). Moreover, we denote by $$\Catt (-,-) : \Cat ^\op\times \Cat\to \Cat $$ the composition of $\Cat[-,-]$
with the functor that gives the underlying discrete category. Finally, a \textit{small category}
is a category $\SSSSS $ such that the underlying discrete category,  \textit{i.e.} $\Catt (\mathsf{1}, \SSSSS ) $,
and the collection of morphisms, \textit{i.e.} $\Catt (\mathsf{2}, \SSSSS ) $, 
 consist of sets. Equivalently,  a \textit{small category} is an internal category of $\Set $.

A $2$-category herein is the same as a $\Cat $-enriched category. We denote the \textit{enriched hom}
of a $2$-category $\A $ by
$$\A (-,-): \A ^\op \times \A \to \Cat $$
which, again, is of course a $2$-functor.
As usual, 
the composition of $1$-cells (morphisms) are denoted
by $\circ $, $\cdot $ or omitted whenever it is clear from the context. 
The vertical composition of $2$-cells is denoted by $\cdot $  or omitted when it is clear,  while the horizontal composition is denoted
by $\ast$. From the vertical and horizontal compositions, we construct the fundamental operation of \textit{pasting}~\cite{MR1040947}, introduced in \cite[Section~1]{MR0220789} and \cite[pag.~79]{MR0357542}.

We denote by $\Delta  $ the full subcategory of the underlying category of $\Cat $ whose objects are finite nonempty 
ordinals seen as posets (or thin categories). We are particularly interested in the subcategory $\Delta _\mathrm{3} $ of $\Delta $
with the objects $\mathsf{1}, \mathsf{2} $ and $\mathsf{3} $ 
generated by the morphisms
%
\pu
\begin{equation*}
	\diag{cosimplicial_onetwothree}
\end{equation*}
with the following relations:
\begin{equation*}
	\begin{aligned}
			s^0 d^1 &= \id _{\mathsf{1} } = s^0 d^0;
	\end{aligned}
	\qquad\qquad\qquad
	\begin{aligned}
			D ^t d ^k &=& D^{k}d^{t-1},\mbox{ if } t>k. 
	\end{aligned}
\end{equation*}

In order to fix notation, we briefly recall the definition of pseudofunctor between a category $\CCCCC $ and a
$2$-category  $\A $ below. For the case of $\A = \Cat $, this
definition was originally introduced by Grothendieck~\cite{MR0354651}
in its contravariant form, while its further generalization for
arbitrary bicategories was originally introduced by B\'{e}nabou~\cite[Section~4]{MR0220789} under the name 
\textit{homomorphism of bicategories}. 

\begin{defi}\label{pseudofunctorcategory2category}
Let $\CCCCC $ be a category (which can be seen as a locally discrete $2$-category) and $\A $ a $2$-category. A \textit{pseudofunctor} $\FFFF :\CCCCC \to \A $ is a pair $(\FFFF , \ffff ) $ with the following data:
\begin{itemize}
\renewcommand\labelitemi{--}	
\item A \textit{function} $\FFFF : \obj (\CCCCC )\to \obj (\A ) $;
\item For each pair $(x,y)$ of objects in $\CCCCC $, functors $\FFFF _{x,y} : 
\CCCCC (x,y)\to \A (\FFFF (x), \FFFF(y)) $, in which $\CCCCC (x,y) $ is seen
as a discrete category;
\item For each pair $g: x\to y , h: y\to z $ of morphisms in $\CCCCC $, an
invertible $2$-cell in $\A $: 
$$\ffff _ {hg}: \FFFF (h) \FFFF(g)\Rightarrow \FFFF(hg); $$
\item For each object $x$ of $\CCCCC $, an invertible $2$-cell in $\A $: $$\ffff_ x: \id _{\FFFF (x)} \Rightarrow \FFFF(\id _ x  );$$
\end{itemize} 
	such that, if $g: x\to y,  h:y\to z$ and  $ e:w\to x  $ are morphisms of $\CCCCC $, the following equations hold in $\A $:
	\begin{enumerate}
		\item Associativity:
		$$\xymatrix{ 
			\FFFF w\ar[rr]^{\FFFF (e)}\ar[dd]_{\FFFF (hge)}\ar[ddrr]|{\FFFF (ge)}
			&&
			\FFFF x\ar[dd]^{\FFFF (g)}\ar@{}[dl]|{\xLeftarrow{\ffff _ {ge}}} 
			&&
			\FFFF w\ar[rr] ^{ \FFFF (e) }\ar[dd]_{\FFFF (hge)}\ar@{}[dr]|{\xLeftarrow{\ffff_ {(hg)e}}}
			&&
			\FFFF x\ar[dd]^{\FFFF (g)}\ar[ddll]|{ \FFFF (hg)}
			\\
			&&&=&&&
			\\
			\FFFF z\ar@{}[ru]|{\xLeftarrow{\ffff _ {h(ge)}}}
			&&
			\FFFF y\ar[ll]^{\FFFF (h)}
			&&
			\FFFF z 
			&&
			\FFFF y\ar[ll]^{\FFFF (h)} \ar@{}[ul]|{\xLeftarrow{\ffff _ {hg}}} 							
		}$$
		\normalsize
		\item Identity:
		$$\xymatrix{  \FFFF w     \ar[rr]^{\FFFF (e)}\ar[dd]_{\FFFF (\id _ {x}e)}&&
			\FFFF x\ar@/_5ex/[dd]|{\FFFF (\id _ {x}) }
			\ar@{}[dd]|{\xLeftarrow{\ffff _ {x}} }
			\ar@/^5ex/[dd]|{\id _ {\FFFF x} }
			&&
			\FFFF w\ar[dd]_{\FFFF (e\id _ {w})}
			&& 
			\FFFF w\ar@{=}[ll]\ar@/_5ex/[dd]|{ \FFFF (\id _ {{}_w}) }
			\ar@{}[dd]|{\xLeftarrow{\ffff _ {{}_w}} }
			\ar@/^5ex/[dd]|{\id _ {\FFFF w} }
			&&
			\FFFF w\ar@/_4ex/[dd]|{\FFFF (e) }
			\ar@{}[dd]|{=}
			\ar@/^4ex/[dd]|{\FFFF (e) }										
			\\
			&\ar@{}[l]|{\xLeftarrow{\ffff _ {\id _ {x} e }}} &
			&=&
			&\ar@{}[l]|{\xLeftarrow{\ffff _ {{}_{e\id _ {w}  }}}} &
			&=& \\										
			\FFFF x\ar@{=}[rr]&&
			\FFFF x
			&&
			\FFFF x &&\FFFF w\ar[ll]^{\FFFF (e)}																				
			&& \FFFF x							}$$ 
		\normalsize			
	\end{enumerate}								
\end{defi}
In this paper, we are going to be particularly
interested in pseudofunctors of the type
$$(\AAAA , \aaaa): \Delta _{\mathrm{3}}\to\A ,$$
also called truncated pseudocosimplicial objects.
For simplicity, given such a truncated 
pseudocosimplicial
category, we define:
\begin{equation*}
	\begin{aligned}
		\AAAA (\sigma _ {01} ) &=& 
		\aaaa _{{}_{D^0 d^0}}^{-1}\cdot 
		\aaaa _{{}_{D^1 d^0}}; & \\
		\AAAA (\sigma _ {02} )
		 &=& \aaaa _{{}_{D^0 d^1}}^{-1}\cdot 
		 \aaaa _{{}_{D^2 d^0}}; &  \\
		\AAAA (\sigma _ {12} ) 
		&=& 
		\aaaa _{{}_{D^1 d^1}}^{-1}\cdot 
		\aaaa _{{}_{D^2 d^1}}; &
	\end{aligned}
	\qquad\qquad\qquad
	\begin{aligned}
			\AAAA (\mathfrak{n}_0 )  &= 
			\aaaa _{{}_{\mathsf{1} }}^{-1}\cdot 
			\aaaa _{{}_{s^0 d^0}}; \\
		\AAAA (\mathfrak{n}_1 )  &= 
		\aaaa _{{}_{\mathsf{1} }}^{-1}\cdot 
		\aaaa _{{}_{s^0 d^1}}. 
	\end{aligned}
\end{equation*}
Using this terminology, we recall the definition of the lax descent category of a pseudofunctor $\Delta _{\mathrm{3}}\to\Cat $.

\begin{defi}[Lax descent category]\label{definitionlaxdescentcategory}
Given a pseudofunctor $(\AAAA , \aaaa): \Delta _{\mathrm{3}}\to\Cat $, the \textit{lax descent category}  $\laxdea  $
of $\AAAA$ is defined as follows:
\begin{enumerate}
	\item The objects are pairs 
	$(w,\varphi )$ in which $w$ is an object of $ \AAAA (\mathsf{1} ) $ and $$\varphi : \AAAA (d ^1)(w)\to\AAAA (d ^0)(w) $$ is a morphism in $ \AAAA (\mathsf{2} ) $ satisfying the following equations:	
	\begin{itemize}
		\item[] Associativity:
		$$\AAAA (D ^0 )(\varphi )\,\cdot \, \AAAA (\sigma _ {02}) _ {w}\, \cdot \, \AAAA (D ^2)(\varphi ) = \AAAA (\sigma _ {01} )_ {w}\,\cdot\,\AAAA(D ^1)(\varphi )\,\cdot\,\AAAA (\sigma _ {12}) _ {w};   $$
		\normalsize
		\item[] Identity:
		$$\AAAA(\mathfrak{n} _0) _ {w}\,\cdot\, \AAAA(s^0) (\varphi ) = \AAAA(\mathfrak{n} _1) _ {w}. $$
		\normalsize
	\end{itemize}
	If the pair $(w,\varphi )$ is an object of $\laxdea $, we say that $\varphi $ is a \textit{descent datum} for $w$ w.r.t. $\AAAA $, or just an $\AAAA $-\textit{descent datum} for $w$. 	
	\item A morphism $\mathfrak{m} : (w,\varphi )\to (w',\varphi ') $ is a morphism $\mathfrak{m}: w\to w' $ in $\AAAA (\mathsf{1} ) $ such that $$\AAAA (d ^0)(\mathfrak{m} )\cdot\varphi  =\varphi ' \cdot \AAAA (d ^1)(\mathfrak{m} ) .$$
\end{enumerate}
The composition of morphisms is given by the composition of morphisms in $\AAAA (\mathsf{1} ) $.
\end{defi}

The lax descent category comes with an obvious \textit{forgetful functor} 
\begin{eqnarray*}
	\dd ^\AAAA: &\laxdea &\to \AAAA (\mathsf{1} )\\
	& (w, \varphi ) &\mapsto w\\
	&\mathfrak{m} &\mapsto \mathfrak{m} 
\end{eqnarray*}
and a natural transformation 
$\uppsi : \AAAA (d^1)\circ \dd ^\AAAA \Longrightarrow \AAAA (d^0)\circ \dd ^\AAAA $ pointwise
defined  by 
$$\uppsi _{(w, \varphi)} : = \varphi : \AAAA (d^1)(w)\to \AAAA (d^0)(w) .$$

Actually, the pair $\left( \dd ^\AAAA: \laxdea\to \AAAA (\mathsf{1} ), \uppsi : \AAAA (d^1)\circ\dd ^ \AAAA \Rightarrow \AAAA (d^0) \circ \dd ^ \AAAA
\right) $ 
 is a two dimensional limit  of $\AAAA $ (see \cite[pag.~177]{MR0401868} or, for instance, in our cause of pseudofunctors, \cite[Section~3]{2016arXiv160703087L}). Namely,
 the lax descent category of $$(\AAAA , \aaaa): \Delta _{\mathrm{3}}\to\Cat $$
 is the \textit{lax descent object}, as defined below, of the pseudofunctor $\AAAA$ in the $2$-category $\Cat $.

 \begin{defi}[Lax descent object~\cite[Section~2]{2019arXiv190201225L}]\label{definitionoflaxdescent}
 Given a pseudofunctor $\AAAA : \Delta _{\mathrm{3}}\to\A $,  
 the \textit{lax descent object}
of $\AAAA $ is, if it exists, an object $\laxdea $ of $\A$ together with a pair
%
\pu
%
\pu
\begin{equation*}
\left(\diag{universalonecelllaxdescentda}, \quad\diag{universaltwocelllaxdescentpsi}\right)
\end{equation*} 
of a morphism $\dd ^\AAAA : \laxdea\to \AAAA (\mathsf{1} ) $,
 called herein the \textit{forgetful morphism} (of descent data),  and
 a $2$-cell $\uppsi $ satisfying the following universal property.
 \begin{enumerate} 
 	\item For each pair
 	$\left( F: \s\to \AAAA (\mathsf{1} ),\,\beta : \AAAA ( d^1  )\circ F\Rightarrow \AAAA ( d^0  )\circ F \right) $ 
 	in which $F$ is a morphism and $\beta $ is a $2$-cell such that  
 	the equations
%
\pu
%
\pu
\begin{equation}\label{Associativityequationdescent}
	\diag{laxdescentassociativityleftside}\quad =\quad \diag{laxdescentassociativityrightside}
\end{equation}
%
\pu
%
\pu
\begin{equation}\label{Identityequationdescent}
\diag{laxdescentidentityleftside}\quad = \quad \diag{laxdescentidentityrightside}
\end{equation}
hold in $\A $, there is a unique morphism  
\begin{equation}
\check{F}: \s\to \laxdea
\end{equation}
in $\A $
making the equations
%
\pu
%
\pu
\begin{eqnarray}\label{firstequationofF}
F&=&\dd ^{\AAAA }\circ \check{F}\\
\diag{equation_on_the_descent_datum_leftside} & = &\diag{equation_on_the_descent_datum_rightside}
\end{eqnarray}
hold. 	
In this case, we say that the
 	 $2$-cell $\beta $
 	is an \textit{$\AAAA $-descent datum} for the morphism $F$.

\item The pair $(\dd ^\AAAA, \uppsi ) $ satisfies the \textit{descent associativity} (Equation~\eqref{Associativityequationdescent}) and the \textit{descent identity} (Equation~\eqref{Identityequationdescent}). In this case, the unique morphism
 	induced is clearly the identity on $\laxdea $.
 	
\item \label{universalpropertyofthelaxdescentfortwocells} Assume that 
$$F_1', F_0': \s \to \laxdea $$
are morphisms of $\A $. For each $2$-cell $\xi : \dd^{\AAAA}\,\circ\, F_1' \Rightarrow \dd^{\AAAA}\,\circ\, F_0'  : \s\to \AAAA (\mathsf{1} ) $ satisfying the equation 
%
\pu
%
\pu
\begin{equation*}
\diag{equation_two_cell_for_descent_left_side_2}\quad=\quad\diag{equation_two_cell_for_descent_right_side_2}
\end{equation*}

there is a unique $2$-cell  $$\xi ' : F_1' \Rightarrow F _0' : \s\to\laxdea  $$  
 	such that
%
\pu
\begin{equation*}
	\diag{xilaxdescenttwocellproperty}=\quad \xi .
\end{equation*}

 \end{enumerate}
 \end{defi}

\begin{lem}
	Let  $\AAAA : \Delta _{\mathrm{3}}\to\A $ be a pseudofunctor.	
	The pseudofunctor $\AAAA $ has a lax descent object $\laxdea $ if and only if 
	there is an isomorphism
	$$\A \left( \s , \laxdea \right)\cong \laxde\left(
	\A \left(\s,\AAAA - \right) \right) $$
	$2$-natural in $\s $,
	in which 
	$\A \left(\s,\AAAA - \right) : \Delta _{\mathrm{3}}\to \Cat $
	is the composition below. 
	$$\xymatrix@=5em{
		\Delta _{\mathrm{3}}\ar[r]^-{\AAAA }\ar@/_2pc/[rr]|-{\A \left(\s,\AAAA - \right) }
		&
		\A\ar[r]^-{\A \left(\s, - \right)}
		&
		\Cat 	
	}$$
\end{lem}

\section{Forgetful morphisms and Kan extensions}\label{Sectionforgetfulmorphismsandkanextensions}
Assuming the existence of the lax descent object of a pseudofunctor 
$(\AAAA , \aaaa): \Delta _{\mathrm{3}}\to\A $, the
forgetful morphism 
 $\dd ^\AAAA $ has many properties that are direct consequences of 
 the definition. Among them, the morphism $\dd ^\AAAA $ is
 \textit{faithful} and \textit{conservative} (by which we mean that, for any object $\s $ of $\A $,   the functor $\A (\s , \dd ^\AAAA )$ is faithful and reflects isomorphisms).

In this section,  we give the core observation of the present paper. Namely,
we investigate the properties of creation of Kan extensions by $\dd ^\AAAA $. 
We start by briefly recalling the basic definitions of  preservation and creation
of Kan extensions (see, for instance, \cite[Section~I.4]{MR0280560} and \cite[Section~2]{MR0299653}).

Let $J: \s\to \cC  $ and $H: \s\to \b $ be morphisms of a $2$-category $\A $. The \textit{right Kan extension of $J $ along $H $} is, if it exists, 
the right reflection $\ran _ H\, J $ of $J$ along the functor
$$\A (H, \cC ) : \A ( \b ,  \cC )\to \A (\s , \cC ). $$
This means that the right Kan extension is actually a pair 
$$\left( \ran _ H\, J : \b\to \cC , \gamma  : \left(\ran _ H\, J\right)\circ H\Rightarrow J 
\right) $$ 
consisting of a morphism $\ran _ H\, J $ and a $2$-cell $\gamma $, called the universal $2$-cell, in $\A$  such that, for each morphism $R: \b\to \cC $ of $\A $,
%
\pu	
%
\pu	
\begin{equation}
\diag{definition_of_kan_extension_leftside}\quad \mapsto \quad\diag{definition_of_kan_extension_rightside}
\end{equation}
defines a bijection $\A (\b ,\cC )(R, \ran _ H\, J)\cong \A (\s , \cC )(R \circ H,  J) $.

Let $J: \s \to \cC $, $H: \s\to \b $ and $G: \cC\to \d $ be morphisms in $\A $.
If $(\hat{J} ,\gamma ) $ is the right Kan extension  of 
$J$ along $H$, we say that $G$ \textit{preserves the right Kan extension
$\ran _H \, J$ } 
if the pair
%
\pu	
\begin{equation*}
	\left( G\circ \hat{J}, \, \diag{preservationofkanextension}\right)  
\end{equation*}
is the right Kan extension $\ran _ H\,  GJ $ of $GJ$ along $H$. 
Equivalently,  $G$ preserves $\ran _ H\, J $ 
if $\ran _ H\, G J $  exists
and, in addition to that, the unique $2$-cell 
$$ G\circ \hat{J} \Rightarrow  \ran _ H \, G J, $$ 
induced by the pair $( G\circ \hat{J} ,  \id _ G\ast \gamma )$ and the universal property of  $\ran _ H \, G J $, 
is invertible (see, for instance, a discussion on canonical (iso)morphisms in \cite{arXiv:1711.02051}).

Furthermore, we say that $G$ \textit{reflects the right Kan extension of $J$ along $H$} if, whenever $( G\circ \hat{J}, \id _ G\ast \gamma  ) $ is the right Kan extension of $GJ $ along $H$, $(\hat{J},  \gamma  ) $ is the
right Kan extension  of $J$ along $H$.

Finally, assuming the existence of $\ran _H\, GJ $, we say that $G: \cC\to \d $ \textit{creates the right Kan extension of 
$G J : \s\to \d   $ along $H$} if we have that (1) $G$ reflects $\ran _H\, GJ $ and (2)
 $\ran _H\, J $ exists and is preserved by $G$.

\begin{rem}[Left Kan extension]
Codually, we have the notion of \textit{left Kan extension} of a morphism  $J: \s\to \cC  $ along $H: \s\to \b $, denoted herein by $\lan _H J$. We also have the appropriate codual definitions of those introduced above. Namely, the concepts of \textit{preservation}, \textit{reflection} and \textit{creation} of left Kan extensions.
\end{rem}

\begin{rem}[Conical (co)limits]\label{trivialRemarkforconicallimits} 
For $\A = \Cat $, 
right Kan extensions along functors of the type $\SSSSS\to\mathsf{1} $
give the notion of conical limits. This is the most elementary and 
well known relation between Kan extensions
and conical limits, which gives 
the most elementary examples of right Kan extensions. We briefly recall
this fact below (see, for instance, \cite[Section~4]{MR0280560}).

Let $  J : \SSSSS\to \CCCCC $ be a functor in which $\SSSSS $ is a small category.
Recall that a cone over $J$ is a pair 
\begin{equation*}
\left(
w,
\vcenter{\xymatrix{
		&
		\SSSSS
		\ar@{}[dd]|-{\xRightarrow{\hskip 2em \kappa\hskip 2em } }
		\ar[ld]
		\ar@/^3pc/[dd]^-{J}
		\\
		\mathsf{1}
		\ar[rd]_-{w}
		&
		\\
		&\CCCCC }}
\right)
\end{equation*}
in which $\mathsf{1}$ is the terminal category, $w: \mathsf{1}\to \CCCCC $ denotes the functor whose image is the object $w$, and $\kappa$
is a natural transformation.

Denoting the composition of 
$$\xymatrix{\SSSSS \ar[r] & \mathsf{1}\ar[r]^{w}  & \CCCCC  } $$
by $\overline{w} $,  a morphism $\iota : w\to w' $ of $\CCCCC$
defines a morphism between the cones $(w, \kappa : \overline{w}\Rightarrow J )   $ 
and $(w ', \kappa ' : \overline{w '}\Rightarrow J )   $ 
 over $J$ if the equation
\begin{equation*} 
\vcenter{\xymatrix{
		&
		\SSSSS
		\ar@{}[dd]|-{\xRightarrow{\hskip 2em \kappa\hskip 2em } }
		\ar[ld]
		\ar@/^3pc/[dd]^-{J}
		\\
		\mathsf{1}
		\ar[rd]_-{w}
		&
		\\
		&\CCCCC }}
\quad =\quad
\vcenter{
\xymatrix{
\SSSSS
\ar[d]
\ar@/^4pc/[dd]^-{J}
\ar@/^2pc/@{{}{ }{}}[dd]|-{\xRightarrow{\hskip 0.5em \kappa '\hskip 0.5em } }
\\
\mathsf{1}
\ar@/_2pc/@{{}{ }{}}[d]|-{\xRightarrow{\hskip 0.5em \iota\hskip 0.5em } }
\ar@/_4pc/[d]_-{w }
\ar[d]^-{w' } 
\\
\CCCCC 
}} 
\end{equation*}
holds, in which, by abuse of language,  $\iota  $ denotes the natural transformation defined by the morphism
$\iota : w\to w' $. 

The above defines a category of cones over $J$. If it exists, the \textit{conical limit
of $J$} is the terminal object of this category. This is clearly equivalent to saying that
 the conical limit of $J$, denoted herein by $\lim J $, is the right Kan extension
$\ran _{\SSSSS\to \mathsf{1}}\, J $ in the $2$-category of categories $\Cat $, either one existing if the other does. In this context, the definitions of \textit{preservation}, \textit{reflection} and \textit{creation} of conical limits coincide 
with those coming from the respective notions in the case of right Kan extensions along $\SSSSS\to\mathsf{1} $ (\textit{e.g.} \cite[Section~4]{MR0280560}). 

Codually, the notion of \textit{conical colimit} of $J: \SSSSS \to \CCCCC $ coincides with
the notion of left Kan extension of $J$ along the unique functor $\SSSSS\to \mathsf{1} $ 
in the $2$-category $\Cat $. Again, the notions of \textit{preservation}, \textit{reflection} and \textit{creation} of conical colimits
coincide with those coming from the respective notions in the case of left Kan extensions along $\SSSSS\to \mathsf{1} $. 

It is well known that there is a deeper relation between conical (and weighted) limits and Kan extensions for much more general contexts. For instance, in the case of  $2$-categories endowed with Yoneda structures~\cite{MR0463261}, the concept of pointwise Kan extensions encompasses this relation (\textit{e.g.} \cite[pag.~50]{MR0280560} for the original case of the $2$-category of $\mathbb{V}$-enriched categories).
Although this concept plays a fundamental role in the theory of Kan extensions, we do not give further comment since we do not use this concept in the present paper. 
\end{rem}

In order to prove our main theorem, we present an elementary 
result below, whose version for limits and colimits is well known.
 
\begin{lem}\label{lemmaforconservativemorphisms}
Let $\A $ be a $2$-category and $H, J, G $ morphisms of $\A$. Assume that
$\ran _H J: \b\to \cC $ exists and is preserved by $G: \cC\to \d $.
If $G$ is conservative, then $G$ creates  the right Kan extension of 
$GJ$ along $H$.
\end{lem}
\begin{proof}
By hypothesis, 
$(G\cdot \ran _H J, \id _ G\ast \gamma ) $ is the right Kan extension of $GJ$ along $H$.
If  $(G\cdot \check{J}, \id _ G\ast \gamma ' ) $ is also the right Kan extension of $GJ$ along $H$,
we get a (unique) induced invertible $2$-cell $G\cdot \check{J}\Rightarrow G\cdot \ran _H J $. By the uniqueness property,
this induced invertible $2$-cell should be the image by 
$\A (\s , G )$ of the $2$-cell
$ \check{J}\Rightarrow \ran _H J $
induced by the universal property of $\ran _H J  $ and the $2$-cell $\gamma ' $. Since $\A (\s , G )$ reflects isomorphisms, the proof is complete.
\end{proof}

\begin{theo}[Main Theorem]\label{maintheorem}
Assume that the lax descent object of 
the pseudofunctor
$(\AAAA , \aaaa): \Delta _{\mathrm{3}}\to\A $  exists.
Let 
$  J : \s\to \laxdea $ and $ H: \s \to \b  $ be morphisms of $\A $. 
\begin{itemize}
	\renewcommand\labelitemi{--}	
\item The forgetful morphism $\dd ^\AAAA : \laxdea \to \AAAA(\mathsf{1} )$ creates the right Kan extension 
of $\dd ^\AAAA  J$ along $H$, provided that $\ran _ H\left( \dd ^\AAAA  J\right) $ 
exists and is preserved by $\AAAA (d^0)$ and $ \AAAA (D^0) \cdot \AAAA (d^0)$.
\item Codually, the forgetful morphism $\dd ^\AAAA : \laxdea \to \AAAA(\mathsf{1} )$ creates the left Kan extension 
of $\dd ^\AAAA  J$ along $H$, provided that  
$\lan _ H\left( \dd ^\AAAA  J\right) $ 
exists and is preserved by $\AAAA (d^1)$ and $ \AAAA (D^2) \cdot \AAAA (d^1)$. 
\end{itemize}
\end{theo}
\begin{proof}
By Lemma \ref{lemmaforconservativemorphisms}, since $\dd ^\AAAA $ is conservative, in order to prove that
$\dd ^\AAAA $  creates
the right Kan extension of 
 $\dd ^\AAAA  J : \s\to  \AAAA (\mathsf{1}) $
 along $H$,
it is enough to prove that $\ran _ H\, J$ exists and is preserved
by $\dd ^\AAAA $.

Let 
$\left( \dd ^\AAAA , \uppsi \right)$ be the universal pair
that gives the lax descent object $\laxdea $.	
We assume that $\left( J : \s\to \laxdea , \,  H: \s \to \b \right)  $ is a pair of morphisms in $\A $ such that the right Kan extension 
$$ \left(\ran _H\, \dd ^\AAAA  J,\quad \nu : \left( \ran _H \, \dd ^\AAAA  J\right)\circ H \Rightarrow \dd ^\AAAA J  \right) $$
of $\dd ^\AAAA J $ along $H$ is preserved by $\AAAA (d^0)$ and $ \AAAA (D^0) \cdot \AAAA (d^0)$.
\begin{itemize}
	\renewcommand\labelitemi{--}	
	\item By the universal property of the right Kan extension
%
\pu	
\begin{equation*}
\left( \AAAA (d^0 )\cdot
\ran_H\left(\dd{^\AAAA}{J}\right), \, \diag{preservedkanextension}\right)  
\end{equation*}
we get that there is a unique $2$-cell  $$\varphi : \AAAA (d^1)\cdot \ran _H \,\left( \dd ^\AAAA J\right)
\Rightarrow \AAAA (d^0)\cdot \ran _H \, \dd ^\AAAA J  $$ in $\A $ such that the equation
%
\pu
%
\pu
\begin{equation}\label{definitionofvarphi}
\diag{definitionofphileftside}\quad = \quad \diag{definitionofphirightside}
\end{equation}
	holds. 	We prove below that $(\ran _H \left(\dd ^\AAAA J \right), \varphi ) $ 
	satisfies the  \textit{descent associativity} (Eq.~\eqref{Associativityequationdescent}) and the \textit{descent identity} (Eq.~\eqref{Identityequationdescent}) w.r.t. $\AAAA $.
	
By the \textit{definition of $\varphi $} (see Eq.~\eqref{definitionofvarphi}), we have that
%
\pu
\begin{equation*}
\varphi '  \coloneqq\quad\diag{definitionofvarphilinha}
\end{equation*}
is equal to
%
\pu
\begin{equation*}
\diag{definitionofvarphilinhalogoigual}
\end{equation*}
Since $\uppsi $ is an $\AAAA $-descent datum for $\dd ^\AAAA $, we have that the $2$-cell above (and hence $\varphi ' $) is equal to
%
\pu
\begin{equation*}\label{psiisadescentdatumeq}
\diag{psiisadescentdatum}
\end{equation*}
which, by the  definition of $\varphi$ (see Eq.~\eqref{definitionofvarphi}), is equal to
the $2$-cell 
%
\pu
\begin{equation*}
	\diag{consequence_of_applying_definition_of_varphi_added_step}
\end{equation*}
which, by the  definition of $\varphi$ (see Eq.~\eqref{definitionofvarphi}) again, is equal to
%
\pu
\begin{equation}\tag{$\varphi '' $} 
\diag{consequence_of_applying_definition_of_varphi}
\end{equation}
denoted by $\varphi '' $. It should be noted that we proved that $\varphi ' = \varphi '' $.

By the universal property of the right Kan extension
%
\pu	
\begin{equation*}
\left( \AAAA (D^0)\cdot \AAAA (d^0 )\cdot
\ran_H\left(\dd{^\AAAA}{J}\right), \, \diag{preservedkanextensiondzeroDzero}\right)  
\end{equation*}
the equality $\varphi ' = \varphi '' $ implies that
the \textit{descent associativity}  w.r.t. $\AAAA $ (Eq.~\eqref{Associativityequationdescent})  for the pair
$(\ran _H \, \dd ^\AAAA J , \varphi ) $ holds. 
 
Analogously, we have that, by the \textit{definition of $\varphi $} (see Eq.~\eqref{definitionofvarphi}),
the equation
%
\pu
%
\pu
\begin{equation*}
\diag{leftsidevarphiidentity}\quad = \quad \diag{rightsidevarphiidentity}
\end{equation*}
holds. Moreover, by the \textit{descent identity}   w.r.t. $\AAAA $ (see Eq.~\eqref{Identityequationdescent}) for the pair 
$\left(\dd ^\AAAA , \uppsi \right) $, 
the right side (hence both sides) of the equation above 
is equal to $\nu $. 

Therefore, by the universal property of the
right Kan extension $(\ran_H\left(\dd{^\AAAA}{J}\right), \nu ) $,
we conclude that the \textit{descent identity}  (Eq.~\eqref{Identityequationdescent})  w.r.t. $\AAAA $ for the pair 
$(\ran_H\left(\dd{^\AAAA}{J}\right), \varphi ) $ holds.

This completes the proof that 
$\varphi $ is an $\AAAA$-\textit{descent datum} for $\ran_H\left(\dd{^\AAAA}{J}\right) $.

\item By the universal property of the lax descent object, we conclude that
there is a unique morphism $\check{J}: \b\to \laxdea $ of $\A $ such that
%
\pu
%
\pu
\begin{eqnarray*}
\ran_H\left(\dd{^\AAAA}{J}\right) &=&\dd ^{\AAAA }\circ \check{J}\\
\\	\diag{equation_on_the_descent_datum_leftside_KanExtension} & = &\diag{equation_on_the_descent_datum_rightside_KanExtension}
\end{eqnarray*}
Moreover, by the universal property of the lax descent object (taking $\xi = \nu $ in Eq.~\ref{universalpropertyofthelaxdescentfortwocells} of Definition \ref{definitionoflaxdescent}) and Equation~\eqref{definitionofvarphi}, it follows that there is a unique $2$-cell
$\tilde{\upnu} : \check{J}\cdot H\Rightarrow J $ in $\A $ such that
$$ \id _{\dd ^\AAAA } \ast \tilde{\upnu }  = \nu .
$$
We prove below that the pair $(\check{J}, \tilde{\upnu } )$ is in fact the
right Kan extension of $J$ along $H$. 

Given any morphism
$R: \b \to \laxdea $ and any $2$-cell 
%
\pu	
\begin{equation}
\diag{omegaforthekanextension}
\end{equation}
of $\A $, by the universal property of the right Kan extension
$$\left(\ran_H\left(\dd{^\AAAA}{J}\right),\, \nu \right) = \left( \dd ^\AAAA \cdot \check{J} ,\,   \id _{\dd ^\AAAA } \ast \tilde{\upnu }\right),$$ there is a unique $2$-cell
$$ \beta : \dd ^\AAAA \circ R \Rightarrow \ran_H\left(\dd{^\AAAA}{J}\right) $$ in $\A $ such that
%
\pu
%
\pu
%
\pu
\begin{equation*}
	\diag{equationKanextensioninordertodefinebetaomega} = \diag{equationKanextensioninordertodefinebetauniversalproperty} = 
\diag{equationKanextensioninordertodefinebetauniversalpropertyvvlinha}
\end{equation*}
It should be noted that, by the definition of $\beta $, we have that
%
\pu
%
\pu
%
\pu
\begin{equation}\label{betalinha}
\beta'\coloneqq\diag{deAomegaintermsofbetanutildeuppsi}=\diag{deAomegaintermsofbetanutildeuppsirightside} 
\end{equation}
holds. Again, by the definition of $\beta $, the right side of Equation~\eqref{betalinha} is equal to
\begin{equation*}
\beta '' \coloneqq \diag{rightsidetoseethatbetaisatwocell}
\end{equation*}
which proves that $\beta ' = \beta '' $.  

By the universal property of the right Kan extension 
$$\left( \AAAA (d^0)
\cdot \ran_H\left(\dd{^\AAAA}{J}\right) ,\, \id _{\AAAA (d^0)}\ast \nu \right) \,\, = \,\, \left( \AAAA (d^0)
\cdot \dd^\AAAA\cdot \check{J} , \, \id _{\AAAA (d^0)}\ast\id _{\dd^\AAAA}\ast \tilde{\upnu}  \right), $$
since $\beta ' = \beta '' $, we conclude that
%
\pu
%
\pu
\begin{equation*}
\diag{leftsideoftheequationforbetalaxdescent}\, =\, \diag{rightsideoftheequationforbetalaxdescent}
\end{equation*}
By the universal property of $\laxdea $ (see \ref{universalpropertyofthelaxdescentfortwocells} of Definition \ref{definitionoflaxdescent}), we get that there is a unique 
$2$-cell 
$\tilde{\upbeta} :  R\Rightarrow \check{J} $ in $\A$
such that 
$$\id_{\dd ^\AAAA}\ast  \tilde{\upbeta} = \beta .$$
By the faithfulness of $\dd ^\AAAA $, it is clear then that
$\tilde{\upbeta} $ is the unique $2$-cell such that
%
\pu	
%
\pu	
\begin{equation*}
\diag{kan_extension_rightside_prova_nocasodenu}\quad = \quad \diag{omegaforthekanextensionlinha} 
\end{equation*}
This completes the proof that $(\check{J} , \tilde{\upnu } ) $ is the right Kan extension of
$J $ along $H$. 

\item Finally, from the definition of  $\ran_H J = (\check{J} , \tilde{\upnu } ) $, it is clear that $\ran_H J  $ is indeed
preserved by $\dd ^\AAAA $.
\end{itemize}

\end{proof}

\subsection{Creation of absolute Kan extensions}

In a $2$-category $\A $, we say that a right Kan extension $\ran _ H J $ 
is \textit{absolute} if it is preserved by any morphism whose domain is 
the codomain of $\ran _ H J $. Moreover, we say that a morphism $G $ \textit{creates absolute right Kan extensions}
if, whenever $\ran _H GJ $ is an absolute right Kan extension, $G$ creates it.
We, of course, have evident codual notions for \textit{left Kan extensions}.

Finally, we say that \textit{$G$ creates absolute Kan extensions}
if it creates both absolute right Kan extensions and absolute left Kan extensions.

The following is an immediate consequence 
of Theorem \ref{maintheorem}.
\begin{coro}\label{absolutekanextensionsanddescent}
	Assume that 
	$(\AAAA , \aaaa): \Delta _{\mathrm{3}}\to\A $  has a lax descent object.
The forgetful morphism $\dd ^\AAAA : \laxdea\to \AAAA (\mathsf{1})$ creates  absolute Kan extensions.\\
Consequently, if a morphism $F$ of $\A $
is equal to $\dd ^\AAAA $ composed with any equivalence, then $F$
 creates 
absolute Kan extensions. 
\end{coro}
\begin{proof}
By Theorem \ref{maintheorem}, the morphism $\dd ^\AAAA $ creates all the right Kan extensions preserved by $\AAAA (d^0)$ and $ \AAAA (D^0) \cdot \AAAA (d^0)$. Since absolute right Kan extensions with codomain in $\AAAA{(\mathsf{1})}$ are preserved by $\AAAA (d^0)$ and $ \AAAA (D^0) \cdot \AAAA (d^0)$, we have that 
$\dd ^\AAAA $ creates absolute right Kan extensions. Codually,  	$\dd ^\AAAA $ creates absolute left Kan extensions. This completes the proof $\dd ^\AAAA $  creates absolute Kan extensions.

To prove the second statement, it should be noted that equivalences create all Kan extensions. Hence, whenever $\mathfrak{F} $ is an equivalence, the composition $F = \mathfrak{F}\circ \dd ^\AAAA $ creates
any of the Kan extensions that are created by   $\dd ^\AAAA $. In particular, $F$ creates absolute Kan extensions. 
\end{proof}

Finally, as a consequence of Remark \ref{trivialRemarkforconicallimits}  and Corollary \ref{absolutekanextensionsanddescent}, since the notion of absolute limits/colimits of diagrams $J: \SSSSS\to \CCCCC $
coincides with the notion of absolute right/left Kan extensions along $\SSSSS\to \mathsf{1}$,
we get:

\begin{coro}\label{Catmaintheoremconsequence}
Let
	$(\AAAA , \aaaa): \Delta _{\mathrm{3}}\to\Cat $  be a pseudofunctor.
If a functor $F$ 
is equal to $\dd ^\AAAA $ composed with any equivalence, then $F$
creates 
absolute limits and colimits. 
\end{coro}

Recall that, since split coequalizers, also called split forks, are examples of absolute coequalizers (see the Proposition on \cite[pag.~224]{MR2056584}),
\textit{a right adjoint functor is monadic if, and only if, it creates absolute coequalizers} (by Beck's monadicity theorem~\cite{MR1987896}).
 Therefore, by Corollary \ref{Catmaintheoremconsequence} we get:

\begin{theo}\label{MONADIC} 
	Assume that $G: \BBBBB\to \CCCCC $ has a left adjoint. If  
	there is a pseudofunctor $(\AAAA , \aaaa): \Delta _{\mathrm{3}}\to\Cat $  
	such that 
	$G =\dd ^\AAAA \circ \mathcal{K} $ for some equivalence $\mathcal{K} $, then $G$ is monadic.
\end{theo}
\begin{proof}
	Assume that $G$ has a left adjoint.
	
If there is a pseudofunctor $(\AAAA , \aaaa): \Delta _{\mathrm{3}}\to\Cat $ such that $G=\dd ^\AAAA \circ \mathcal{K} $ for an equivalence $\mathcal{K}$, then $G$ creates absolute coequalizers
	by Corollary \ref{Catmaintheoremconsequence}. By Beck's monadicity theorem, we conclude that $G$ is monadic.
\end{proof}	

Codually, we have:
\begin{theo}\label{COMONADIC}
	Assume that $G: \BBBBB\to \CCCCC $ has a right adjoint. If $G$ is equal to a forgetful functor of descent data $\dd ^\AAAA$ composed with an equivalence, then $G$ is comonadic.
\end{theo}

\begin{rem}\label{monadicityinCat}
If $G$ is monadic, then,
by the monadicity theorem of \cite[Section~5]{2019arXiv190201225L}, $G$ is
an \textit{effective faithful functor}. That is to say, $G$ is the forgetful functor of descent data w.r.t. its \textit{$2$-dimensional cokernel diagram}
$$ (\BBBB , \bbbb): \Delta _{\mathrm{3}}\to\Cat $$ 
composed with an equivalence  
(see \cite[Section~2]{2019arXiv190201225L} for the definition of the \textit{$2$-dimensional cokernel diagram} of a morphism).

Therefore, by the above and Theorem \ref{MONADIC},  we have:
\begin{theo}[Monadicity Theorem]\label{monadicityinCat2}
Assume that $G: \BBBBB\to \CCCCC $ has a left adjoint. The functor $G$ is monadic if, and only if,
there is a pseudofunctor $(\AAAA , \aaaa): \Delta _{\mathrm{3}}\to\Cat $  
such that 
$G =\dd ^\AAAA \circ \mathcal{K} $ for some equivalence $\mathcal{K} $.
\end{theo}
\end{rem}

\begin{rem}[Creation of limits and of absolute colimits]
	We do not give full definitions in this remark, since it is not the
	main point of this paper. The interested reader 
	may find the missing definitions and proofs in \cite{2019arXiv190201225L}.
	
	Employing the monadicity theorem of \cite[Section~5]{2019arXiv190201225L}, Theorem \ref{maintheorem} can be seen as a generalization of the
	well known results of creation of limits (and colimits)
	of monadic functors.
	
	More precisely, as mentioned in the proof of Theorem \ref{MONADIC}, by 
	the monadicity theorem of \cite[Section~5]{2019arXiv190201225L},
	given a monadic functor $G$, denoted herein by $(\BBBB , \bbbb): \Delta _{\mathrm{3}}\to\Cat $,
	 we get that $G$ is the forgetful morphism (of the descent data) w.r.t. its \textit{$2$-dimensional cokernel diagram}. Therefore:
	\begin{itemize}
		\renewcommand\labelitemi{--}
		\item 	Since $G $ has a left adjoint, $\BBBB  (d^0) $ and 
	$\BBBB  (D^0)\cdot \BBBB  (d^0) $ have left adjoints (see \cite[Section~4]{2019arXiv190201225L}). Hence, since right adjoint morphisms preserve all right Kan extensions, we get that $G$ creates all right Kan extensions by Theorem
	\ref{maintheorem}. In particular, $G$ creates all limits.
	
		\item By Theorem \ref{maintheorem}, we have that $G$, being a forgetful morphism of descent data, does create all
		absolute left Kan extensions. But, more generally, $G$ creates the left Kan extensions that are preserved by $\BBBB (D^2)\cdot \BBBB (d^1) $ and $\BBBB (d^1) $. 
		
		Therefore, by the definition of the $2$-dimensional cokernel diagram $\BBBB $ (\cite[Section~2]{2019arXiv190201225L}), we conclude that
	$G$ creates the left Kan extensions that are preserved by $T$ and $T^2 $ (in which $T$ denotes the endofunctor underlying the codensity monad of $G$).
	\end{itemize}	
\end{rem}

\section{Descent theory}\label{Sectiondescenttheory}
Using the concepts previously introduced in this paper, we briefly recover the classical setting of descent theory w.r.t. fibrations. The exposition in this section is heavily influenced by \cite{MR1466540, 2016arXiv160604999L} but it is independent and different from both of them.

Following their approach, instead of considering fibrations, 
we start with a pseudofunctor
$$\FFFF: \CCCCC^\op \to \Cat 
$$
which can be also called an \textit{indexed category}.

A \textit{precategory} in $\CCCCC $ is a functor
$a: \Delta_{\mathrm{3} }^\op \to \CCCCC $ and, hence, each 
internal category or groupoid of $\CCCCC $ 
 has an underlying precategory. In particular, internal groups and monoids
 w.r.t. the cartesian structure also have underlying precategories.
By abuse of language, whenever a precategory $a$ is the underlying precategory
of an internal category (internal groupoid, monoid or group), we say that the precategory $a$ \textit{is} 
an internal category (internal groupoid, monoid or group).

\begin{rem}[Composition of pseudofunctors]
Let	$a: \Delta_{\mathrm{3} }^\op \to \CCCCC $ be a precategory. Firstly,
we can consider the functor $\op (a): \Delta_{\mathrm{3} } \to \CCCCC ^\op $, also denoted by
$a ^\op $, 
which is the image of $a$ by the usual dualization (invertible) $2$-functor
$$\op : \Cat ^\co\to\Cat .$$
Secondly, we can consider that $\op (a): \Delta_{\mathrm{3} } \to \CCCCC ^\op $ is actually a pseudofunctor between locally discrete $2$-categories.
Therefore we can define the composition
$$\FFFF\circ\, \op (a): \Delta_{\mathrm{3} }\to \Cat $$
as a particular case of composition of pseudofunctors/homomorphisms
of bicategories/$2$-categories. Namely,
the composition is defined by
\begin{eqnarray*}
\FFFF\circ\, \op (a) := \BBBB : & \Delta_{\mathrm{3} } & \to \Cat\\ 
& x & \mapsto \FFFF\left(a(x) \right) \\
& g: x\to y & \mapsto \FFFF\left(a^\op (g: x\to y ) \right)   \\
 \bbbb _  {x} : = \ffff _{ a(x) } : &\id _ {\FFFF (a(x) ) } &\Rightarrow \FFFF\left(  \id _{a(x) }\right) \\
\bbbb _  {hg}  : = \ffff _ {a^\op (h) a^\op (g) } : &\FFFF \, a^\op (h) \cdot \FFFF \, a^\op (g)    &
\Rightarrow \FFFF \, a^\op (hg).
\end{eqnarray*}
\end{rem}	
By definition, the category of \textit{$\FFFF $-internal actions} of a precategory $a: \Delta_{\mathrm{3} }^\op \to \CCCCC $ (actions $a\to \CCCCC $)
 is the lax descent object of the composition
$\FFFF\circ\, \op(a)  : \Delta_{\mathrm{3} }\to \Cat $. That is to say,
$$\FFFF\textrm{-}\textrm{IntAct}\left(a  \right) := \laxde\left(\FFFF\circ \,\op (a)  \right) .
$$
As briefly mentioned in the introduction, this definition 
generalizes the well known definitions of categories of actions. For instance, taking $\CCCCC = \Set $ and $\FFFF = \Set /- : \Set ^\op \to\Cat $,
if $a: \Delta_{\mathrm{3} } ^\op\to \Set $ is an internal category, 
the category of $\left( \Set /-\right)$-internal actions of $a $ coincides
up to equivalence with the category $\Cat \left[ a, \Set \right] $ of functors $a\to\Set $ and natural transformations. This shows that the definition above has as particular cases
the well known categories of $m$-sets (or $g$-sets) for a monoid $m$ (or a group $g$).

Analogously, given a topological group $g$, we can consider the category of 
$g\textrm{-}\Top $ of the Eilenberg-Moore algebras of the monad $g\times - $ with the
multiplication $g\times g \times - \to g\times - $ given by the operation of $g$, that is to say, 
the category of $g$-spaces. This again coincides with the category of
$\left(\Top /-\right) \textrm{-}\textrm{IntAct}\left( g  \right)  $, in which $g$, by abuse of language,
is the underlying precategory of $g$.

A precategory is \textit{discrete} if it is naturally isomorphic to a constant functor $\overline{w} : \Delta_{\mathrm{3} }^\op \to \CCCCC $ for an object $w$ of $\CCCCC $.
Clearly, we have:
\begin{lem}
	The category of $\FFFF$-internal actions  of 
	a discrete precategory $\overline{w} $ is equivalent to $\FFFF (w) $.
\end{lem}

Given a precategory $a : \Delta_{\mathrm{3} }^\op\to \CCCCC $,  the \textit{underlying discrete
precategory} of the
precategory $a$ is the precategory constantly equal to $a (\mathsf{1}) $, which we  denote by $\overline{a (\mathsf{1})} : \Delta_{\mathrm{3} }^\op\to \CCCCC $. 
We have, then, that the functor
$$\laxde\left(\FFFF\circ a^\op \right)\to \FFFF\circ a (\mathsf{1} )
$$
that forgets the descent data is the forgetful functor 
$$\FFFF\textrm{-}\textrm{IntAct}\left( a  \right)\to
\FFFF\textrm{-}\textrm{IntAct}\left(\overline{a(\mathsf{1})} \right)
$$
between the category of $\FFFF$-internal actions of $a$ and the 
category of $\FFFF$-internal actions of the underlying discrete precategory of $a$.

\begin{rem}[Underlying discrete precategory]\label{discretizationremark}
	The definition of the
	\textit{underlying discrete precategory} of a precategory  is 
	motivated by the special case of internal categories,
	and/or the case of precategories that can be extended to truncated simplicial objects
	$\underline{\Delta_{\mathrm{3} }}^\op\to \CCCCC $,
	\begin{equation*}
	\vcenter{
		\xymatrix{  \mathsf{1}\ar[rr]|-{s_0} && \mathsf{2}
			\ar@/^1.5pc/[ll]|-{d_0}\ar@/_1.5pc/[ll]|-{d_1}
			\ar@/_2.5pc/[rr]|-{S_0}\ar@/^2.5pc/[rr]|-{S_1}
			&& \mathsf{3}\ar@/^1.3pc/[ll]|-{D_0}\ar[ll]|-{D_1}\ar@/_1.3pc/[ll]|-{D_2} } }
	\qquad\to\qquad
	\CCCCC ,
	\end{equation*} 
	in which $\underline{\Delta}_{\mathrm{3} } $ is the full subcategory of $\Delta $ with the objects
	$\mathsf{1}$, $\mathsf{2}$ and $\mathsf{3}$.
	We have an adjunction
	$$\xymatrix{
		\Cat \left[\underline{\Delta}_{\mathrm{3} } ^\op, \CCCCC\right] \ar@/_0.8pc/[rr]\ar@{}[rr]|{\bot }
		&&
		\Cat \left[\mathsf{1} , \CCCCC\right]\cong \CCCCC \ar@/_0.8pc/[ll]
	} $$
	in which the left adjoint is given by the usual functor $w\mapsto \overline{w} $ that associates each 
	object to the constant functor $\overline{w} : \underline{\Delta}_{\mathrm{3} }\to \CCCCC $. Of course, the right adjoint is given by the conical limit, which, in this case,
	coincides with $a(\mathsf{1})$, since $\mathsf{1}$ is the initial object of $\underline{\Delta_{\mathrm{3} }}^\op $. The underlying discrete precategory, in this case, is given by the monad induced by this adjunction. 
\end{rem}

\begin{rem}[Forgetful functor]
As particular case of Remark \ref{discretizationremark},
in the case of $\CCCCC = \Set $ and $\FFFF =\left( \Set /-\right) $,  if $a : \Delta_{\mathrm{3} }^\op \to\Set $ is an internal category, 
the forgetful functor $$\left(\Set /-\right) \textrm{-}\textrm{IntAct}\left( a  \right)\to
\left(\Set /-\right) \textrm{-}\textrm{IntAct}\left(\overline{a(\mathsf{1})} \right)
$$ coincides with the usual forgetful functor
$\Cat [a, \Set ]\to  \Set ^{a(\mathsf{1})}\simeq \Set /a(\mathsf{1})  $ between the category of functors
$a\to\Set $ and the category of \textit{functions} between the set  $a(\mathsf{1})$
of objects of $a $ and the collection of objects of $\Set $. In particular, this shows that, if $a$ is a monoid, we get that this forgetful functor 
coincides with the usual forgetful functor $a\textrm{-}\Set\to \Set $.
Analogously, taking $\CCCCC = \Top $ and $\FFFF = \left(\Top /- \right) $, if $g: \Delta_{\mathrm{3} }^\op \to \Top$
is an internal group (topological group), then the forgetful functor  	
$$\left(\Top /-\right) \textrm{-}\textrm{IntAct}\left( g  \right)\to
\left(\Top /-\right) \textrm{-}\textrm{IntAct}\left(\overline{g(\mathsf{1})} \right)
$$
coincides with the usual forgetful functor $g\textrm{-}\Top\to\Top $ between the category
of $g$-spaces and $\Top$. 
\end{rem}

As a consequence of Corollary \ref{Catmaintheoremconsequence}: 

\begin{coro}
Given an indexed category 	$\FFFF: \CCCCC^\op \to \Cat $
and a precategory $a : \Delta_{\mathrm{3} }^\op\to \CCCCC $,
the forgetful functor 
$$\FFFF\textrm{-}\textrm{IntAct}\left( a  \right)\to
\FFFF\textrm{-}\textrm{IntAct}\left(\overline{a(\mathsf{1})} \right)
$$
creates absolute Kan extensions and, hence, in particular, it creates absolute limits and colimits. 
\end{coro}

\textit{Henceforth, we assume that $\CCCCC $ has pullbacks, and 
a pseudofunctor $\FFFF :\CCCCC ^\op\to \Cat  $ is given.}
Every morphism $p: e\to b$ of $\CCCCC  $
induces an internal groupoid whose underlying precategory, denoted herein
by $\mathsf{Eq}(p) $, is given by
%
\pu
\begin{equation*}
\diag{eqpinternalgroupoid}
\end{equation*}
in which $e\times _b e $ denotes the pullback of $p$ along itself, and  the arrows are given
by the projections and the diagonal morphism (see, for instance, 
\cite[Section~3]{MR1466540}).
For short, \textit{we denote by $\FFFF ^p : \Delta_{\mathrm{3} }\to \Cat $ 
the pseudofunctor obtained by the composition 
$$\FFFF \circ  \mathsf{Eq}(p)^\op : \Delta_{\mathrm{3} }\to \Cat .  $$}

\begin{lem}\label{factorizationdescentfacets}
Let $\left( \dd ^{\FFFF ^p  } , \uppsi \right)$ be the universal pair
that gives the lax descent category of $\FFFF ^p$.  
For each morphism $p: e\to b $ of $\CCCCC $, we get a factorization 
\begin{equation}\label{descentfactorizationp}
\xymatrix{
\FFFF (b)
\ar[rd]|-{K_{ p}}
\ar[rr]^-{\FFFF (p)}
&&
\FFFF (e)
\\
&
\laxde \left(\FFFF ^p\right) 
\ar[ru]|-{\dd ^{{ }^{\FFFF ^p}}}
&
}
\end{equation}
\textit{called the $\FFFF$-descent factorization of $\FFFF (p) $}, 
in which  $K_{p}$ is 
the unique functor such that the diagram above is commutative and the equation
%
\pu
%
\pu
\begin{equation*}
\diag{equation_on_the_descent_datum_leftsideforthecaseofpseudofunctor}\quad =\quad \diag{equation_on_the_descent_datum_rightsideforthecaseofthepseudofunctor}
\end{equation*}	
holds.
\end{lem}
\begin{proof}
This factorization can be found, for instance, in \cite[Section~3]{MR1466540} or \cite[Section~8]{2016arXiv160604999L}.
In our context, in order to prove this result, it is enough to 
verify that 
$$\ffff_{\pi^e\, p }^{-1}\cdot \ffff_{\pi_e\, p } : \FFFF^p (d^1)\cdot \FFFF (p)\Rightarrow \FFFF^p (d^0)\cdot \FFFF (p)  $$
is an $\FFFF ^p $-descent datum for $\FFFF (p)$, which follows directly from the fact that
$\FFFF : \CCCCC ^\op\to\Cat $ is a pseudofunctor.
\end{proof}

\begin{defi}[Effective descent morphism]
A morphism $p$ of $\CCCCC $ is of \textit{effective $\FFFF $-descent} if
the comparison $K_p $ of \textit{the $\FFFF$-descent factorization of $\FFFF (p) $} (Eq.~\eqref{descentfactorizationp}) is an equivalence. 
\end{defi}

\begin{rem}
	By definition, if $p$ is of effective $\FFFF $-descent, this means in particular 
	that $\FFFF (p) : \FFFF (b)\to \FFFF ( e) $ is, up to the composition with a canonical equivalence, 
	the forgetful functor between the category of $\FFFF $-internal actions
	of the internal groupoid $\mathsf{Eq}(p) $ and the category of $\FFFF $-internal actions of the underlying discrete groupoid $\overline{e} $.

\end{rem}

\section{Effective descent morphisms and monadicity}\label{sectionfinale}
The celebrated B\'{e}nabou-Roubaud theorem (see \cite{MR0255631} or, for instance, \cite[Theorem~1.4]{2016arXiv160604999L}) gives
an insightful connection between monad theory and descent theory. Namely,
the theorem says that the $\FFFF$-descent factorization of $\FFFF (p) $ (Eq.~\eqref{descentfactorizationp}) coincides up to equivalence with the Eilenberg-Moore
factorization of the right adjoint functor $\FFFF (p ) $,
provided that $\FFFF $ comes from a bifibration satisfying
the so called \textit{Beck-Chevalley condition}.

The theorem motivates what is often called \textit{monadic approach to descent} (\textit{e.g.} \cite[Section~2]{MR1285884}), and it is useful to the characterization of effective descent morphisms
in  several cases of interest (\textit{e.g.} \cite{MR2056587,MR1285884,MR1271335}).

In our context, the B\'{e}nabou-Roubaud theorem can be stated as follows. 
Assume that $\FFFF : \CCCCC^\op\to \Cat $ is a pseudofunctor such that, for every morphism $p$ of $\CCCCC $, 
\begin{itemize}\renewcommand\labelitemi{--}
\item there is an adjunction $(\FFFF (p)!\dashv \FFFF (p), \varepsilon ^p , \eta ^p) :\FFFF (b)\to\FFFF (e) $, and
\item the $2$-cell obtained from the pasting
\begin{equation*}
\xymatrix{
\FFFF (e)
\ar@{=}[dd]
\ar[rr]^-{\FFFF (p)! }
&&
\FFFF (b)
\ar@{{}{ }{}}@/_2pc/[lldd]|-{\xRightarrow{\hskip 1em\eta ^p \hskip 1em }}
\ar[lldd]|-{\FFFF (p) }
\ar[dd]^-{\FFFF (p) }
\\
& &
\\
\FFFF (e) 
\ar@{{}{ }{}}[rr]|-{\xRightarrow{\ffff_{\pi^e\, p }^{-1}\cdot \ffff_{\pi_e\, p } } }
\ar[dd]_-{\FFFF (\pi ^e ) }
&&
\FFFF (e)
\ar@{{}{ }{}}@/^2pc/[lldd]|-{\xRightarrow{\hskip 1em\varepsilon ^{\pi _e } \hskip 1em }}
\ar[lldd]|-{\FFFF (\pi _e) }
\ar@{=}[dd]
\\
&&
\\
\FFFF (e\times _b e)
\ar[rr]_-{\FFFF (\pi _e)!  }
&&
\FFFF (e)
}
\end{equation*}
is invertible. 
\end{itemize}
We have that, denoting by $T^p $ the monad  $\left(\FFFF (p)\cdot\FFFF (p)!, \id _ {\FFFF (p)} \ast \varepsilon ^p\ast \id _ {\FFFF (p)!}, \eta ^p  \right)$, 
the Eilenberg-Moore factorization 
\begin{equation*}
\xymatrix{
\FFFF (b)\ar[rr]|-{\FFFF (p) }
\ar[rd]
&&
\FFFF (e)\\
&
\FFFF (e)^{ T^p}
\ar[ru]
&	
}
\end{equation*}
is \textit{pseudonaturally} equivalent to the $\FFFF$-descent factorization of $\FFFF (p) $ 
(Eq.~\eqref{descentfactorizationp}). In particular, we get that, assuming the 
above, a morphism \textit{$p$ is of effective $\FFFF $-descent if and only if $\FFFF (p) $ 
is monadic}.

\begin{rem}[Basic bifibration]
	If 
	$\CCCCC $ has pullbacks, 
	the basic indexed category $$\CCCCC /- : \CCCCC ^\op\to \Cat  $$
	satisfies the Beck-Chevalley condition. Therefore, in this case,
	by the B\'{e}nabou-Roubaud theorem, one reduces the
	problem of characterization of effective descent morphisms to the problem of characterization of the morphisms $p$ for which
	the change of base functor $\CCCCC /p = p^\ast : \CCCCC  /b \to \CCCCC /e $ is monadic. 
	
	For instance, if $\CCCCC $ is locally cartesian closed and has coequalizers, one can easily prove
	that $\CCCCC /p $ is monadic if and only if $p$ is a universal regular epimorphism, by Beck's (pre)monadicity theorem
	(\textit{e.g.} \cite[Corollary~1.3]{MR1271335}). 
	 This result on locally cartesian closed categories also plays an essential role in the usual
	framework to study effective $\left(\CCCCC /-\right) $-descent morphisms
	of  more general categories via embedding
	results (see, for instance, \cite[Section~1.6]{MR1271335}, \cite[Section~2]{MR1285884} and \cite[Section~1]{2016arXiv160604999L}).
\end{rem}

\subsection{Indexed categories not satisfying the Beck-Chevalley condition}
	The B\'{e}nabou-Roubaud theorem answers the question of the comparison
	of the Eilenberg-Moore factorization with the $\FFFF$-descent factorization of $\FFFF (p) $ (Eq.~\eqref{descentfactorizationp}) 
	whenever $\FFFF $ satisfies the Beck-Chevalley condition.
	One might ask what it is possible to prove in 
	this direction without assuming
	the Beck-Chevalley condition. 
	
	It should be noted that there
	are indexed categories $\FFFF : \CCCCC ^\op\to\Cat $ (coming from bifibrations that do not satisfy the Beck-Chevalley condition) 
	for which there are non-effective descent morphisms inducing monadic functors.

	For instance, in \cite[Example~3.2.3]{Melo} (\textit{Exemplo} 3.2.3 of pag.~67), Melo gives a detailed
	proof that the so called   \textit{fibration
	of points} of the category of groups does not satisfy the Beck-Chevalley condition (in particular, w.r.t. the morphism $0\to \mathsf{S}_3 $). 
It is known that, denoting by $\mathsf{Pt} $ the corresponding 
indexed category, $\mathsf{Pt} (0\to \mathsf{S}_3) $ is monadic but
$0\to \mathsf{S}_3 $ is not of effective $\mathsf{Pt}$-descent.

We can produce examples of non-effective descent morphisms
inducing monadic functors as above, once we observe that:

\begin{prop}\label{exemplo}
	If the domain of a morphism $p$ is the terminal object  of $\CCCCC $, then
$p$ is of effective $\FFFF $-descent if and only if $\FFFF (p)$ is
an equivalence.
\end{prop} 
\begin{proof}
	Indeed, if the domain of $p$ is the terminal object $1$ of $\CCCCC $, 	
		$\mathsf{Eq}(p)$ is discrete, naturally isomorphic to the 
	precategory $\Delta_{\mathrm{3}} ^\op\to \CCCCC $ constantly 
	equal to $1 $. Consequently, 
	$$\FFFF\textrm{-}\textrm{IntAct}\left(\mathsf{Eq}(p)\right)
	\simeq \FFFF ( 1 ).
	$$	
	Therefore the result follows, since the $\FFFF$-descent factorization of $\FFFF (p) $, in this case, is given by
	\begin{equation*}
	\xymatrix{\FFFF (b)
	\ar[rd]|-{K_{ p}}
	\ar[rr]^-{\FFFF (p)}
	&&
	\FFFF (1)
	\\
	&
	\FFFF\textrm{-}\textrm{IntAct}\left(\mathsf{Eq}(p)\right)
	\ar[ru]|-{\simeq }
	&
	}
	\end{equation*}
\end{proof}	

\begin{rem}
	Proposition \ref{exemplo} gives us a way to study \cite[Example~3.2.3]{Melo}. 
	
	In an \textit{exact protomodular} category (\textit{e.g.} \cite{MR1615341}), denoting again by $\mathsf{Pt} $
	the indexed category corresponding to the fibration of points, whenever $\mathsf{Pt} (p) $ has a left adjoint, it is monadic (see \cite[Theorem~3.4]{MR1615341}).

	In the case of the category of groups, 
	$\mathsf{Pt} (0\to \mathsf{S}_3) $ has a left adjoint but it is not an equivalence. 
	
	Therefore, by  \cite[Theorem~3.4]{MR1615341}, the functor $\mathsf{Pt} (0\to \mathsf{S}_3) $ is monadic, while, by Proposition \ref{exemplo}, the morphism $0\to \mathsf{S}_3$ is not of effective $\mathsf{Pt} $-descent.
\end{rem}

\begin{rem}
It should be noted that, if $p: 1\to b$ is a morphism of $\CCCCC $ satisfying the hypothesis of 
Proposition \ref{exemplo}, the pasting
\begin{equation*}
\xymatrix@=1.5em{
	\FFFF (1)
	\ar@{=}[dd]
	\ar[rr]^-{\FFFF (p)! }
	&&
	\FFFF (b)
	\ar@{{}{ }{}}@/_1.8pc/[lldd]|-{\xRightarrow{\hskip 1em\eta ^p \hskip 1em }}
	\ar[lldd]|-{\FFFF (p) }
	\ar[dd]^-{\FFFF (p) }
	\\
	& &
	\\
	\FFFF (1) 
	\ar@{{}{ }{}}[rr]|-{= }
	\ar[dd]_-{\FFFF (\pi ^1 ) =  \id_ {\FFFF (1)} }
	&&
	\FFFF (1)
	\ar@{{}{ }{}}@/^2.2pc/[lldd]|-{=}
	\ar[lldd]|-{\FFFF (\pi _1) = \id_ {\FFFF (1)}  }
	\ar@{=}[dd]
	\\
	&&
	\\
	\FFFF (1\times _b 1) = \FFFF (1)
	\ar[rr]_-{\id_ {\FFFF (1)} }
	&&
	\FFFF (1)
}
\end{equation*}
is invertible if and only if $\eta ^p $ is invertible.
(or, equivalently, $\FFFF (p)! $ is fully faithful).
In other words, \textit{$p: 1\to b $ satisfies the Beck-Chevalley condition w.r.t.
	$\FFFF $ if and only if $\FFFF (p)! $ is fully faithful}.

Assuming that $\FFFF (p)! $ is fully faithful in the situation above, we get that
$\FFFF (p) $ is (pre)monadic if and only if it is an equivalence. That is to say, in this case, we get, by Proposition \ref{exemplo}, that \textit{$\FFFF (p) $ is (pre)monadic if and only if $p$ is of
effective $\FFFF $-descent}. 
\end{rem}

The most elementary examples of non-effective $\FFFF $-descent morphisms
inducing monadic functors can be constructed from Lemma \ref{exemplo2}.
Namely, in order to get our desired example, it is enough to 
consider a pseudofunctor $\GGGG : \mathsf{2} ^\op\to \Cat $ 
whose image of $d $ is a monadic functor which is not an equivalence.
In this case, by Lemma \ref{exemplo2}, we conclude that,
despite $\GGGG (d) $ being monadic, 
$d$ is not of effective $\GGGG $-descent.

\begin{lem}\label{exemplo2}
	Consider the category $\mathsf{2} $ with the only non-trivial morphism
	$d: \mathsf{0}\to\mathsf{1} $. Given a pseudofunctor $\GGGG :\mathsf{2} ^\op\to\Cat $, $d$ is of  effective $\GGGG$-descent if and only if
	$\GGGG (d) $ is an equivalence.
\end{lem}
\begin{proof}
	Again, in this case, $\mathsf{Eq}(d)$ is discrete. We have that
	$\GGGG\textrm{-}\textrm{IntAct}\left(\mathsf{Eq}(d)\right)
	\simeq \GGGG ( \mathsf{0}  )$,	
	and, hence, we get the result.
\end{proof}

	In \cite{MR2107401}, Sobral characterizes the effective $\mathcal{E}$-descent morphisms in the category $\cat $ of small categories, in which
$\mathcal{E}: \cat ^\op\to\Cat $ can be defined by
\begin{eqnarray}\label{definitionofEdiscreteopfibrations}
\mathcal{E}: &\cat ^\op &\to \Cat \nonumber\\
&e  & \mapsto \Cat \left[  e, \Set \right]\\
& p: e\to b & \mapsto \Cat \left[  p, \Set \right]: \Cat \left[  b, \Set \right]\to \Cat \left[  e, \Set \right].\nonumber
\end{eqnarray}	
As a consequence of her characterization, she shows that the functor $h: \mathsf{1}\sqcup \mathsf{1}\to \mathsf{2}$ which is a bijection on objects
(that is to say, induced by $d^0$ and $d^1$) 
is not an effective $\mathcal{E}$-descent morphism, but 
$ \mathcal{E}(h) $ is monadic
 (see \cite[Remark~7]{MR2107401}). In this context, she also \textit{informally} suggests that,
	for the indexed category $\mathcal{E} $, 
	\textit{descent gives ``more information'' than monadicity}.
	We finish this article showing, as an immediate consequence of Theorem \ref{MONADIC}, that this is in fact the case for
	any indexed category. 
	
\begin{theo}[Effective descent implies monadicity]\label{consequenceofthemaintheoreminthecaseofGrothendieckdescent}
	Let $\FFFF : \CCCCC ^\op\to \Cat $ be any pseudofunctor. If $p$
	is of effective $\FFFF $-descent and $\FFFF (p) $ has a left adjoint,
	then $\FFFF (p) $ is monadic.
\end{theo}
\begin{proof}
	It is clearly a particular case of Theorem \ref{MONADIC}.
\end{proof} 

By Theorem \ref{consequenceofthemaintheoreminthecaseofGrothendieckdescent}, given a pseudofunctor $\FFFF : \CCCCC ^\op\to \Cat $ coming from a bifibration, every effective $\FFFF $-descent morphism $p$ induces 
a monadic functor $\FFFF (p) $.

\begin{rem}\label{JanelidzeExample}
It is worth noting that there are pairs $ \left(\FFFF , p\right)  $ such that  $\FFFF $ is a pseudofunctor coming 
from a bifibration and $p$ is an $\FFFF$-descent morphism  not satisfying the Beck-Chevalley condition. 

For instance, considering the pseudofunctor $\mathcal{E}$ defined in \eqref{definitionofEdiscreteopfibrations}, 
the functor 
$$\underline{h}:\mathsf{1}\sqcup \mathsf{1}\to \mathsf{1} $$
is an effective $\mathcal{E}$-descent morphism, since it is a split epimorphism (by  \cite[Theorem~3.5]{MR1466540}). However $\underline{h}$ does not satisfy the Beck-Chevalley condition.
For instance, taking $f: \mathsf{1}\sqcup \mathsf{1}\to \Set $ defined by the pair $(\emptyset, \left\{\emptyset \right\} ) $, we get that 
$$ \lan _{\pi _{\mathsf{1}\sqcup \mathsf{1}} } \left(f\circ \pi^{\mathsf{1}\sqcup \mathsf{1}} \right) : \mathsf{1}\sqcup \mathsf{1}\to \Set
$$
is defined by the pair $(\emptyset, \left\{\emptyset \right\}\sqcup\left\{\emptyset \right\}  ) $, while 
$$\lan_{\underline{h}} \left( f \right) \circ \underline{h} :
\mathsf{1}\sqcup \mathsf{1}\to \Set
$$
is defined by the pair $\left(\left\{\emptyset \right\} ,  \left\{\emptyset \right\} \right)$.

\end{rem}

\section*{Acknowledgments}
I am very grateful to the \textit{Software Technology Research Group} at Utrecht University for the welcoming, supportive, and inspiring environment. Being part of this group has positively influenced me towards working in fundamental and applied research.

I am also grateful to the \textit{Coimbra Category Theory Group} for their   support and insight.
I am specially thankful to Manuela Sobral for our fruitful discussions on descent theory during my research fellowship at the Centre for Mathematics, University of Coimbra, in 2019.  She kindly taught me about the examples given in  \cite[Remark~7]{MR2107401} and \cite[Example~3.2.3]{Melo}
of non-effective descent morphisms inducing monadic functors.

George Janelidze warmly hosted me at the University of Cape Town in Nov/2019, where I started my first revision of this paper. On that occasion, he kindly gave the example for which I was looking, described in Remark \ref{JanelidzeExample}.

This work is part of a project that started during my postdoctoral fellowship at 
the \textit{Centre for Mathematics of the University of Coimbra} and, more particularly, during my research visit at the \textit{Universit\'{e} catholique de Louvain} in May 2018.

I thank Tim Van der Linden for being attentive and considerate while working as the editor of this paper. Finally, I thank the anonymous referee who read the article in detail and came up with regardful and helpful comments and suggestions.

\bibliographystyle{plain}
\bibliography{references}

\pu

\end{document}
